\documentclass[a4paper, 10pt, notitlepage]{article}

\usepackage{amsthm, amsmath, amssymb, latexsym}
\usepackage{mathrsfs,cite}
\usepackage{graphicx}
\usepackage{epstopdf}

\usepackage[ruled,vlined]{algorithm2e}

\usepackage[a4paper]{geometry}		
 
\theoremstyle{plain}
\newtheorem{theorem}{Theorem}[section]
\newtheorem{lemma}[theorem]{Lemma}
\newtheorem{corollary}[theorem]{Corollary}
\newtheorem{proposition}[theorem]{Proposition}

\theoremstyle{definition}
\newtheorem{definition}[theorem]{Definition}
\newtheorem{example}[theorem]{Example}
\newtheorem{assumption}[theorem]{Assumption}

\theoremstyle{remark}
\newtheorem{remark}{Remark}[section]

\usepackage{chngcntr}
\counterwithin{theorem}{section}	
\counterwithin{remark}{section}	

\DeclareMathOperator{\co}{co}
\DeclareMathOperator{\dist}{dist}
\DeclareMathOperator*{\minimise}{minimise}

\author{M.V. Dolgopolik\footnote{Institute for Problems in Mechanical Engineering of the Russian Academy of Sciences,
Saint Petersburg, Russia}
\footnote{This work was performed in IPME RAS and supported by the Russian Science Foundation (Grant No. 20-71-10032).}}
\title{Steering exact penalty DCA for nonsmooth DC optimisation problems with equality and inequality
constraints}

\begin{document}

\maketitle

\begin{abstract}
We propose and study a version of the DCA (Difference-of-Convex functions Algorithm) using the $\ell_1$ penalty function
for solving nonsmooth DC optimisation problems with nonsmooth DC equality and inequality constraints. The method employs
an adaptive penalty updating strategy to improve its performance. This strategy is based on the so-called steering exact
penalty methodology and relies on solving some auxiliary convex subproblems to determine a suitable value of the penalty
parameter. We present a detailed convergence analysis of the method and illustrate its practical performance by
applying the method to two nonsmooth discrete optimal control problem.
\end{abstract}

\section{Introduction}

Since the late 1980s DC (Difference of Convex functions) optimisation has been one of the most popular research areas
in nonconvex and nonsmooth optimisation. An interest in this class of problems is predicated on the fact that one can
apply well-developed apparatus of convex analysis and convex optimisation to efficiently solve problems having DC
structure. The possibility to utilise {D}{C} structure of a problem turned out to be especially fruitful in the
nonsmooth case, and a wide variety of methods for minimising nonsmooth DC functions have been developed over the years.
Among them are codifferential methods \cite{BagirovUgon,TorBagKar}, bundle and double bundle methods
\cite{JokiBagirovKarmitsa,JokiBagirov2018,JokiBagirov2020}, methods based on successive DC piecewise-affine
approximations \cite{GaudiosoBagirov}, aggregate subgradient method \cite{BagirovTaheriJoki}, etc.

Perhaps, the most renown method for solving DC optimisation problems is the so-called \textit{DCA}, originally proposed
by Pham Dinh and Souad in \cite{PhamDinh1986} and later on extensively developed and applied to various particular
classes of problems in the works of Le Thi and Pham Dinh et al.
\cite{PhamDinhLeThi96,DinhLeThi1997,PhamDinhLeThi98,LeThiDinh2005,LeThiPhamDinh2014,LeThiPhamDinh2014b}. 
DC optimisation methods closely related to DCA were studied by de Oliveira et al.
\cite{deOliveiraTcheou,deOliveira2019,AckooijDeOliveira2019b,AckooijDeOliveira2019}.
A detailed literature review on DCA and related topics can be found in \cite{LeThiDinh2018,LippBoyd}. 

Despite the abundance of publications on DC optimisation methods, relatively few papers have been devoted to
development of numerical methods for general \textit{constrained} DC optimisation problems, although some specific
classes of constrained DC optimisation problems (such as convex maximisation problems
\cite{PardalosRosen,Enkhbat,Romeijn}, nonconvex quadratic programs
\cite{CarmonDuchi,LeThiDinh98_Quadratic,CambiniSodini,BurerVandebussche}, etc.) have received much attention of
researchers. 

Bundle-type methods for solving nonsmooth DC optimisation problems with inequality constraints were developed in 
\cite{AckooijDemassey,MontonenJoki}. DCA-type methods for solving such problems were studied in
\cite{LeThiPhamDinh2014,LeThiPhamDinh2014b,AckooijDeOliveira2019,LippBoyd}, while such methods for problems with more
general cone constraints (e.g. semidefinite constraints) were considered in \cite{LippBoyd,Dolgopolik_DCcone}. Some of
the methods from the aforementioned papers require a feasible starting point, while in the case when such point is
unknown, the reader is only referred to general exact penalty methods (cf. \cite{AckooijDeOliveira2019} and
\cite[Sect.~16.1]{Bagirov_book}). In turn, when the exact penalty techniques are employed (as in
\cite{LeThiPhamDinh2014,LeThiPhamDinh2014b,LippBoyd,Dolgopolik_DCcone}), only the simplest penalty updating
strategies are usually used, such as the rule to increase the penalty parameter by a constant factor $\rho > 1$ after
each iteration till a feasible point is found. However, as is illustrated by multiple examples in
\cite{ByrdNocedalWaltz,ByrdLopezCalvaNocedal}, a choice of penalty parameter is an important and difficult problem.
Serious effort must be put into developing efficient penalty updating strategies, since an inadequate (either too small
or too large) value of the penalty parameter might significantly slow down the convergence. 

Finally, let us note that very little research on DC optimisation problems with DC \textit{equality} constraints
exists. To the best of the author's knowledge, only in the recent papers by Strekalovsky 
\cite{Strekalovsky2017,Strekalovsky2020} optimisation methods for general DC optimisation problems with both equality
and inequality constraints have been considered.

The main goal of this paper is to present and analyse a DCA-type exact penalty method with an \textit{adaptive} penalty
updating strategy for solving nonsmooth DC optimisation problems with nonsmooth DC equality and inequality
constraints. Namely, we aim at developing an exact penalty method that takes into account information about computed
points to adaptively adjust the penalty parameter in a way that would improve overall convergence of constructed
sequence.

The method presented in this paper is based on the general steering exact penalty methodology developed for sequential
linear/quadratic programming methods for nonlinear programming problems by Byrd et al. in
\cite{ByrdNocedalWaltz,ByrdLopezCalvaNocedal}. A DCA-type method using steering exact penalty rules for constrained
nonsmooth DC optimisation problems was first presented by Strekalovsky in \cite{Strekalovsky2020}. However, both the
description of this method and its convergence analysis in \cite{Strekalovsky2020} contain several inaccuracies (see 
Remark~\ref{rmrk:StrekalovskyMethod} below for more details). In particular, the case when a point computed by the
method is critical for the penalty term (i.e. constraints are, in a sense, degenerate at this point) is left out of
consideration in \cite{Strekalovsky2020}, which might lead to an incorrect behaviour of the method for some practical
problems.

In this paper, we present a detailed discussion of the steering exact penalty methodology in the context of DC
optimisation problems and use it to develop a correct version of the steering exact penalty DCA. In contrast
to \cite{Strekalovsky2020}, we prove the correctness of our method, that is, we prove that the sequence constructed by
the steering exact penalty DCA is correctly defined and each iteration of the method requires solution of a finite
number of convex optimisation subproblems. Furthermore, we present a much more detailed convergence analysis of the
method than in \cite{Strekalovsky2020} and, in particular, provide simple sufficient conditions for the boundedness of
the penalty parameter in the case when there are no equality constraints (in \cite{Strekalovsky2020} the penalty
parameter is assumed to be bounded). 

The paper is organized as follows. Section~\ref{sect:Criticality} contains a discussion of two notions of criticality
for constrained DC optimisation problems, which are needed for a theoretical analysis of the steering exact penalty
DCA. Section~\ref{sect:ExPenDCA_Steering} is devoted to a detailed description of this method and its convergence
analysis. Finally, two numerical examples illustrating performance of the steering exact penalty DCA are given in
Section~\ref{sect:NumericalExamples}.

\section{Critical points of DC optimisation problems}
\label{sect:Criticality}

Throughout this article we study the following constrained nonsmooth DC optimisation problem:
\begin{align*}
  &\minimise \enspace \: f_0(x) = g_0(x) - h_0(x)
  \\
  &\text{subject to} \enspace 
  \begin{aligned}[t]
    f_i(x) &= g_i(x) - h_i(x) \le 0, \quad i \in \mathcal{I}, \qquad \qquad \qquad \qquad (\mathcal{P})
    \\
    f_j(x) &= g_j(x) - h_j(x) = 0, \quad j \in \mathcal{E}, \quad x \in A.
  \end{aligned}
\end{align*}
Here $g_k, h_k \colon \mathbb{R}^d \to \mathbb{R}$, $k \in \{ 0 \} \cup \mathcal{I} \cup \mathcal{E}$, are given convex
functions, $\mathcal{I} = \{ 1, \ldots, \ell \}$ and $\mathcal{E} = \{ \ell + 1, \ldots, m \}$ are finite index sets
(one of which can be empty), and $A \subseteq \mathbb{R}^d$ is a closed convex set.

Before we proceed to a discussion of exact penalty methods for the problem $(\mathcal{P})$, let us first introduce two
notions of criticality for this problem, which are intimately related to optimality conditions for nonsmooth
mathematical programming problems in terms of Demyanov-Rubinov-Polyakova quasidifferentials
\cite{Dolgopolik_SIAM,Dolgopolik_MetricReg}. For a detailed discussion of optimality conditions and criticality for
unconstrained DC optimisation problems see \cite{LeThiDinh2018,AckooijDeOliveira2019,JokiBagirov2020}.

Let $N_A(x) = \{ v \in \mathbb{R}^d \mid \langle v, y - x \rangle \le 0 \: \forall y \in A \}$ be the normal cone to 
the set $A$ at a point $x \in A$. Here $\langle \cdot, \cdot \rangle$ is the inner product in $\mathbb{R}^d$. 

\begin{definition} \label{def:Criticality}
A feasible point $x_*$ of the problem $(\mathcal{P})$ is said to be \textit{critical} for this problem, if there
exist subgradients $v_k \in \partial h_k(x_*)$, $k \in \{ 0 \} \cup \mathcal{I} \cup \mathcal{E}$, and 
$w_j \in \partial g_j(x_*)$, $j \in \mathcal{E}$, and Lagrange multipliers 
$\lambda_i, \underline{\mu}_j, \overline{\mu}_j \ge 0$, $i \in \mathcal{I}$, $j \in \mathcal{E}$, such that
\begin{equation} \label{eq:Optimality}
\begin{split}
  0 \in \partial g_0(x_*) - v_0
  + \sum_{i \in \mathcal{I}} \lambda_i \big( \partial g_i(x_*) - v_i \big)
  &+ \sum_{j \in \mathcal{E}} \underline{\mu}_j \big( \partial g_j(x_*) - v_j \big)
  \\
  &- \sum_{j \in \mathcal{E}} \overline{\mu}_j \big( w_j - \partial h_j(x_*) \big) + N_A(x_*),
\end{split}
\end{equation}
and the complementarity condition $\lambda_i f_i(x_*) = 0$ holds true for all $i \in \mathcal{I}$.
\end{definition}

\begin{remark}
Arguing in a similar way to the proof of \cite[Thm.~3]{Dolgopolik_MetricReg} (see also \cite{Dolgopolik_SIAM}), one
can check that a locally optimal solution $x_*$ of the problem $(\mathcal{P})$ is critical for this problem, provided a
suitable constraint qualification holds at $x_*$. Several such constraint qualifications are discussed in
\cite{Dolgopolik_MetricReg,Dolgopolik_SIAM}.
\end{remark}

Let us point out an almost obvious, yet useful reformulation of the notion of criticality. For any feasible point $x$
denote $\mathcal{I}(x) = \{ i \in \mathcal{I} \mid f_i(x) = 0 \}$.

\begin{lemma} \label{lem:Criticality}
A feasible point $x_*$ is critical for the problem $(\mathcal{P})$ if and only if there exist $c_* > 0$, 
$v_k \in \partial h_k(x_*)$, $k \in \{ 0 \} \cup \mathcal{I} \cup \mathcal{E}$, and $w_j \in \partial g_j(x_*)$, 
$j \in \mathcal{E}$, such that for any $c \ge c_*$ the point $x_*$ is a global minimiser of the convex function
\begin{multline} \label{eq:GlobalPenMajorant}
  Q_c(x) = g_0(x) - \langle v_0, x - x_* \rangle
  + c \sum_{i \in \mathcal{I}} \max\big\{ g_i(x) - h_i(x_*) - \langle v_i, x - x_* \rangle, 0 \big\}
  \\ 
  + c \sum_{j \in \mathcal{E}} \max\big\{ g_j(x) - h_j(x_*) - \langle v_j, x - x_* \rangle, 
  h_j(x) - g_j(x_*) - \langle w_j, x - x_* \rangle \big\}
\end{multline}
on the set $A$.
\end{lemma}

\begin{proof}
Observe that $Q_c(x_*) = g_0(x_*)$ for any $c > 0$ due to the feasiblity of $x_*$. In addition, the function $Q_c$ is
nondecreasing in $c$. Therefore, $x_*$ is a point of global minimum of $Q_c$ on the set $A$ for some $c > 0$ if and only
if $x_*$ is a point of global minimum of $Q_t$ on $A$ for any $t \ge c$. Thus, it is sufficient to check that a feasible
point $x_*$ is critical for the problem $(\mathcal{P})$ if and only if there exists $c > 0$ such that $x_*$ is a point
of global minimum of $Q_c$ on $A$.

By the standard optimality conditions, $x_*$ is a point of global minimum of the convex function $Q_c$ on the set $A$
for some $c > 0$ if and only if $0 \in \partial Q_c(x_*) + N_A(x_*)$. In turn, by the standard rules of 
the subdifferential calculus this inclusion is satisfied for some $c > 0$ if and only if 
\begin{align*}
  0 \in \partial g_0(x_*) - v_0 &+ c \sum_{i \in \mathcal{I}(x_*)} \co\big\{ \partial g_i(x_*) - v_i, 0 \big\}
  \\
  &+ c \sum_{j \in \mathcal{E}} \co\big\{ \partial g_j(x_*) - v_j, \partial h_j(x_*) - w_j \big\} + N_A(x_*)
\end{align*}
or, equivalently, if and only if there exist $\alpha_i \in [0, 1]$, 
$i \in \mathcal{I}(x_*)$, and $\beta_{j1}, \beta_{j2} \in [0, 1]$, $j \in \mathcal{E}$, such that 
$\beta_{j1} + \beta_{j2} \le 1$, $j \in \mathcal{E}$, and
\begin{equation} \label{eq:CriticalityIntermediate}
\begin{split}
  0 \in \partial g_0(x_*) - v_0 &+ c \sum_{i \in \mathcal{I}(x_*)} \alpha_i \big( \partial g_i(x_*) - v_i \big)
  \\
  &+ c \sum_{j \in \mathcal{E}} \Big( \beta_{j1} \big( \partial g_j(x_*) - v_j \big) 
  + \beta_{j2} \big( \partial h_j(x_*) - w_j \big) \Big) + N_A(x_*).
\end{split}
\end{equation}
Note that it is sufficient to assume that $\beta_{j1} + \beta_{j2} \le 1$ (instead of $\beta_{j1} + \beta_{j2} = 1$),
since one has $0 \in \co\big\{ \partial g_j(x_*) - v_j, \partial h_j(x_*) - w_j \big\}$.

As is easily seen, inclusion \eqref{eq:CriticalityIntermediate} is satisfied if and only if condition
\eqref{eq:Optimality} holds true for some $\lambda_i, \underline{\mu}_j, \overline{\mu}_j > 0$, $i \in \mathcal{I}$, 
$j \in \mathcal{E}$, satisfying the complementarity condition $\lambda_i f_i(x_*) = 0$ for all $i \in \mathcal{I}$.
Indeed, if inclusion \eqref{eq:CriticalityIntermediate} holds true, then one can define $\lambda_i = c \alpha_i$ for
$i \in \mathcal{I}(x_*)$, $\lambda_i = 0$ for $i \notin \mathcal{I}(x_*)$, $\underline{\mu}_j = c \beta_{j1}$ and
$\overline{\mu}_j = c \beta_{j2}$ for all $j \in \mathcal{E}$. Conversely, if condition \eqref{eq:Optimality} holds
true, then one can set 
$c = \max\{ \lambda_i, \underline{\mu}_j + \overline{\mu}_j \mid i \in \mathcal{I}, j \in \mathcal{E} \}$
and define $\alpha_i = \lambda_i / c$ for all $i \in \mathcal{I}$, $\beta_{j1} = \underline{\mu}_j / c$
and $\beta_{j2} = \overline{\mu}_j / c$ for all $j \in \mathcal{E}$. Thus, $x_*$ is a point of global minimum of
the convex function $Q_c$ on the set $A$ for some $c > 0$ if and only if $x_*$ is a critical point of the problem
$(\mathcal{P})$.
\end{proof}

The lemma above allows one to introduce a natural extension of the notion of criticality from
Definition~\ref{def:Criticality} to the case of infeasible points.

\begin{definition}
A point $x_* \in A$ is called \textit{a generalised critical point} of the problem $(\mathcal{P})$ for a given value 
$c > 0$ of the penalty parameter, if there exist $v_k \in \partial h_k(x_*)$, 
$k \in \{ 0 \} \cup \mathcal{I} \cup \mathcal{E}$, and $w_j \in \partial g_j(x_*)$, $j \in \mathcal{E}$, such that $x_*$
is a globally optimal solution of the convex problem
\[
  \minimise \: Q_c(x) \quad \text{subject to} \quad x \in A,
\]
where $Q_c$ is defined as in \eqref{eq:GlobalPenMajorant}.
\end{definition}

\begin{remark} \label{rmrk:GeneralizedCriticality}
From the previous definition and Lemma~\ref{lem:Criticality} it follows that any critical point of the problem
$(\mathcal{P})$ is a generalised critical point of this problem for any sufficiently large value of the penalty 
parameter. Conversely, any \textit{feasible} generalised critical point of the problem $(\mathcal{P})$ is critical for
this problem. Let us also note that the generalised criticality depends on the choice of the penalty parameter $c > 0$.
In some cases one can escape a generalised critical point by simply changing this parameter. See
\cite[Remark~9]{Dolgopolik_DCcone} for a more detailed discussion.
\end{remark}

\section{Steering exact penalty DCA}
\label{sect:ExPenDCA_Steering}

Being inspired by the steering exact penalty methods for nonlinear programming problems
\cite{ByrdNocedalWaltz,ByrdLopezCalvaNocedal}, we present a new exact penalty DCA-type method for solving constrained DC
optimisation problems. The method uses the standard $\ell_1$ penalty function for the problem $(\mathcal{P})$ and
updates its penalty parameter in essentially the same way as the penalty parameter is updated in the steering exact
penalty methods. This approach to penalty updates is based on solving some auxiliary convex subproblems to determine a
suitable value of the penalty parameter that would ensure balanced progress towards both feasibility and optimality. The
need to potentially solve multiple auxiliary subproblems (which might be computationally expensive) only to find a
suitable value of the penalty parameter might seem redundant and inefficient at first glance. However, such approach to
penalty updates leads to a substantial reduction of overall number of iterations and improved robustness of
corresponding methods (see \cite{ByrdNocedalWaltz,ByrdLopezCalvaNocedal} for the results of numerical experiments for
steering exact penalty versions of some SQP-type methods).

\subsection{A description of the algorithm}

Our aim is to design an exact penalty DCA-type algorithm for solving the problem $(\mathcal{P})$ based on the $\ell_1$
penalty function
\begin{equation} \label{eq:ell1_PenaltyTerm}
  \Phi_c(x) = f_0(x) + c \varphi(x), \quad
  \varphi(x) = \sum_{i \in \mathcal{I}} \max\{ f_i(x), 0 \} + \sum_{j \in \mathcal{E}} |f_j(x)|.
\end{equation}
To this end, we utilise the convex majorant of this function of the form
\begin{align*}
  Q_c(x, y, V) &= g_0(x) - \langle v_0, x - y \rangle 
  + c \Big( \sum_{i \in \mathcal{I}} \max\big\{ g_i(x) - h_i(y) + \langle v_i, x - y \rangle, 0 \big\} 
  \\
  &+ \sum_{j \in \mathcal{E}} \max\big\{ g_j(x) - h_j(y) - \langle v_j, x - y \rangle, 
  h_j(x) - g_j(y) - \langle w_j, x - y \rangle \big\} \Big),
\end{align*}
where $x, y \in \mathbb{R}^d$, $V = (v_0, v_1, \ldots, v_m, w_{\ell + 1}, \ldots, w_m)$, 
$v_k \in \partial h_k(y)$, $k \in \{ 0 \} \cup \mathcal{I} \cup \mathcal{E}$, and $w_j \in \partial g_j(y)$, 
$j \in \mathcal{E}$. By the definition of subgradient
\[
  h_i(x) - h_i(y) \ge \langle v_i, x - y \rangle,
  \quad
  g_j(x) - g_j(y) \ge \langle w_j, x - y \rangle.
\]
for any $i \in \{ 0 \} \cup \mathcal{I} \cup \mathcal{E}$ and $j \in \mathcal{E}$. Therefore for all 
$x, y \in \mathbb{R}^d$ one has
\begin{equation} \label{eq:ConvexMajorant_Steering}
\begin{split}
  Q_c(x, y, V) - h_0(y) &\ge g_0(x) - h_0(x) + c \sum_{i \in \mathcal{I}} \max\{ g_i(x) - h_i(x), 0 \}
  \\
  &+ c \sum_{j \in \mathcal{E}} \max\{ g_j(x) - h_j(x), h_j(x) - g_j(x) \} = \Phi_c(x),
\end{split}
\end{equation}
and, moreover, $Q_c(x, x, V) - h_0(x) = \Phi_c(x)$. Thus, $Q_c(\cdot, y, V) - h_0(y)$ is a global convex majorant of
$\Phi_c(\cdot)$.

Exact penalty DCA-type algorithms are based on consecutively solving the penalty subproblem
\begin{equation} \label{prob:ExactPenaltyConvexMajorant}
  \min_x \enspace Q_c(x, x_n, V_n) \quad \text{subject to} \quad x \in A,
\end{equation}
where $V_n = (v_{n0}, v_{n1}, \ldots, v_{nm}, w_{n (\ell + 1)}, \ldots, w_{nm})$, $v_{nk} \in \partial h_k(x_n)$, 
$k \in \{ 0 \} \cup \mathcal{I} \cup \mathcal{E}$, and $w_{nj} \in \partial g_j(x_n)$, $j \in \mathcal{E}$. Problem 
\eqref{prob:ExactPenaltyConvexMajorant} can be viewed as the $\ell_1$ penalty version of the following linearised convex
problem:
\begin{equation} \label{prob:LinearizedProblem}
\begin{split}
  &\minimise_x \enspace g_0(x) - \langle v_{n0}, x - x_n \rangle 
  \\
  &\text{subject to} \enspace 
  \begin{aligned}[t]
    g_i(x) - h_i(x_n) + \langle v_{ni}, x - x_n \rangle &\le 0, \quad i \in \mathcal{I}, \quad x \in A
    \\
    g_j(x) - h_j(x_n) - \langle v_{nj}, x - x_n \rangle &\le 0, \quad j \in \mathcal{E},
    \\
    h_j(x) - g_j(x_n) - \langle w_{nj}, x - x_n \rangle &\le 0, \quad j \in \mathcal{E}.
  \end{aligned}
\end{split}
\end{equation}
Note, however, that this linearised problem might have an empty feasible region (if $x_n$ is infeasible for the problem
$(\mathcal{P})$) and, therefore, have no optimal solutions, while penalty subproblem
\eqref{prob:ExactPenaltyConvexMajorant} always has an optimal solution, provided the penalty function $\Phi_c$ is
coercive on the set $A$ (see Prp.~\ref{prp:WellPosedness} below).

To determine a suitable value of the penalty parameter $c$ for problem \eqref{prob:ExactPenaltyConvexMajorant}, we will
use essentially the same approach as in the steering exact penalty methods
\cite{ByrdNocedalWaltz,ByrdLopezCalvaNocedal}. Namely, introduce the convex function
\begin{multline} \label{def:LinInfeasMeasure}
  \Gamma(x, x_n, V_n) = 
  \sum_{i \in \mathcal{I}} \max\big\{ g_i(x) - h_i(x_n) + \langle v_{ni}, x - x_n \rangle, 0 \big\} 
  \\
  + \sum_{j \in \mathcal{E}} \max\big\{ g_j(x) - h_j(x_n) - \langle v_{nj}, x - x_n \rangle, 
  h_j(x) - g_j(x_n) - \langle w_{nj}, x - x_n \rangle \big\},
\end{multline}
which can be used as an infeasibility measure for problem \eqref{prob:LinearizedProblem} (and a global convex majorant
of the penalty term $\varphi$), and consider the auxiliary convex feasibility subproblem
\begin{equation} \label{prob:OptimalFeasibility}
  \minimise_x \enspace \Gamma(x, x_n, V_n) \quad \text{subject to} \quad x \in A,
\end{equation}
that allows one to compute the optimal level of feasibility of the linearised convex problem
\eqref{prob:LinearizedProblem}. In particular, if the feasible region of problem \eqref{prob:LinearizedProblem} is
nonempty, then the optimal value of problem \eqref{prob:OptimalFeasibility} is zero;  conversely, if the optimal value
of this problem is zero and it has an optimal solution, then the feasible region of problem
\eqref{prob:LinearizedProblem} is nonempty.

\begin{remark}
Problem \eqref{prob:OptimalFeasibility} can obviously be rewritten as the following equivalent convex programming
problem:
\begin{align*}
  &\minimise_{(x, y, z)} \enspace \sum_{i \in \mathcal{I}} y^{(i)} + \sum_{j \in \mathcal{E}} z^{(j)}
  \\
  &\text{subject to} \enspace 
  \begin{aligned}[t]
    g_i(x) - h_i(x_n) - \langle v_{ni}, x - x_n \rangle &\le y^{(i)}, \quad y^{(i)} \ge 0, \quad i \in \mathcal{I}, 
    \quad x \in A,
    \\
    g_j(x) - h_j(x_n) - \langle v_{nj}, x - x_n \rangle &\le z^{(j)}, \quad j \in \mathcal{E},
    \\
    h_j(x) - g_j(x_n) - \langle w_{nj}, x - x_n \rangle &\le z^{(j)}, \quad j \in \mathcal{E}.
  \end{aligned}
\end{align*}
Note that in the case of DC optimisation problems with reverse convex constraints (i.e. $\mathcal{E} = \emptyset$ and 
$g_i(\cdot) \equiv 0$ for all $i \in \mathcal{I}$; see \cite{HillestadJacobsen}) this is just a linear programming
problem. In the case when all constraints of the problem $(\mathcal{P})$ are quadratic, this problem is a convex
quadratically constrained linear programming problem, which can be efficiently solved with the use of interior point
methods. There are many other particular cases in which this problem can be solved fairly efficiently.
\end{remark}

Following the steering exact penalty methodology \cite{ByrdNocedalWaltz,ByrdLopezCalvaNocedal}, one can formulate 
the following general guidelines for updating the penalty parameter:
\begin{enumerate}
\item{If the optimal value of problem \eqref{prob:OptimalFeasibility} is zero, then choose $c > 0$ large enough to make
sure that an optimal solution of the penalty subproblem \eqref{prob:ExactPenaltyConvexMajorant} is feasible
for the linearised problem \eqref{prob:LinearizedProblem}.
}

\item{If the optimal value of problem \eqref{prob:OptimalFeasibility} is positive (i.e. the feasible region of problem
\eqref{prob:LinearizedProblem} is empty), then choose $c > 0$ in such a way that the reduction of the infeasibility
measure $\Gamma(\cdot, x_n, V_n)$ is proportional to the best possible reduction computed via problem
\eqref{prob:OptimalFeasibility}, that is, 
\begin{equation} \label{eq:FeasibilityDecay_Guideline}
  \Gamma(x_n(c), x_n, V_n) - \Gamma(x_n, x_n, V_n) 
  \le \eta_1 \Big( \Gamma(\widehat{x}_n, x_n, V_n) - \Gamma(x_n, x_n, V_n) \Big),
\end{equation}
where $x_n(c)$ is an optimal solution of problem \eqref{prob:ExactPenaltyConvexMajorant}, $\widehat{x}_n$ is an optimal
solution of problem \eqref{prob:OptimalFeasibility}, and $\eta_1 \in (0, 1)$ is a fixed parameter.
}

\item{Finally, if condition \eqref{eq:FeasibilityDecay_Guideline} is satisfied and the reduction of the infeasibility
measure $\Gamma(x_n(c), x_n, V_n) - \Gamma(x_n, x_n, V_n)$ is large, then the penalty parameter $c > 0$ must be chosen
to ensure that the reduction of the penalty function $Q_c$ for the linearised problem \eqref{prob:LinearizedProblem} is
sufficiently large as well. We impose this requirement in the form of the following inequality:
\[
  Q_c(x_n(c), x_n, V_n) - Q_c(x_n, x_n, V_n) 
  \le \eta_2 c \Big( \Gamma(x_n(c), x_n, V_n) - \Gamma(x_n, x_n, V_n) \Big).
\]
where $\eta_2 \in (0, 1)$ is a fixed parameter.
}
\end{enumerate}
The guidelines for updating the penalty parameter listed above are just reformulations of the same guidelines for
SQP-type exact penalty methods for nonlinear programming problems from \cite{ByrdNocedalWaltz,ByrdLopezCalvaNocedal} to
the case of an exact penalty DCA for DC optimisation problems. Note, however, that these guidelines cannot be
implemented directly in the context of DC optimisation problems (especially problems with equality constraints).
Indeed, if the optimal value of problem \eqref{prob:OptimalFeasibility} is zero, then in the general case there might
not exist $c > 0$ such that an optimal solution of the penalty subproblem \eqref{prob:ExactPenaltyConvexMajorant} is
feasible for the linearised problem \eqref{prob:LinearizedProblem}. Similarly, 
if $\Gamma(\widehat{x}_n, x_n, V_n) = \Gamma(x_n, x_n, V_n)$, then in the general case there might not exists $c > 0$
for which inequality \eqref{eq:FeasibilityDecay_Guideline} is satisfied. Therefore we must modify the guidelines to
make sure that they can be applied in the context of DC optimisation problems.

To develop a correct version of steering exact penalty DCA, we propose to base penalty updates not on the optimal value
of the feasibility subproblem \eqref{prob:OptimalFeasibility} (i.e. on the fact whether the feasible region of the
linearised problem \eqref{prob:LinearizedProblem} is empty or not), but on the difference between the infeasibility
measure $\Gamma(x_n, x_n, V_n) = \varphi(x_n)$ of the current iterate $x_n$ and the optimal value of the infeasibility
measure $\Gamma(\widehat{x}_n, x_n, V_n)$. 

Namely, if $\Gamma(\widehat{x}_n, x_n, V_n) = \Gamma(x_n, x_n, V_n)$ (in particular, if the point $x_n$ is feasible for
the problem $(\mathcal{P})$), then one must find $c > 0$ such that the infeasibility measure $\Gamma(x_n(c), x_n, V_n)$
is sufficiently close to the optimal value of the infeasibility measure $\Gamma(\widehat{x}_n, x_n, V_n)$. In turn, if
$\Gamma(\widehat{x}_n, x_n, V_n) < \Gamma(x_n, x_n, V_n)$, then one must find $c > 0$ satisfying inequality 
\eqref{eq:FeasibilityDecay_Guideline}. Bearing in mind these guidelines we arrive at the following version of 
the steering exact penalty DCA given in Algorithmic Pattern~\ref{alg:SteeringPenalty}.

\begin{algorithm}[ht!]	\label{alg:SteeringPenalty}
\caption{Steering Exact Penalty DCA}

\noindent\textbf{Initialization.} {Choose an initial guess $x_0 \in A$, an initial value of the penalty parameter 
$c_0 > 0$, parameters $\varepsilon_{feas} > 0$, and $\eta_1, \eta_2 \in (0, 1)$, and set $n := 0$.
}

\noindent\textbf{Step 1.} {Put $c_+ = c_n$. For all $k \in \{ 0 \} \cup \mathcal{I} \cup \mathcal{E}$ compute 
$v_{nk} \in \partial h_k(x_n)$, for all $j \in \mathcal{E}$ compute $w_{nj} \in \partial g_j(x_n)$, and define
\[
  V_n = (v_{n0}, v_{n1}, \ldots, v_{n m}, w_{n (\ell + 1)}, \ldots, w_{nm}).
\]
Compute a solution $x_n(c_+)$ of the convex problem
\begin{equation} \label{prob:StPen_PenaltySubproblem}
  \minimise_x \: Q_c(x, x_n, V_n) \quad \text{subject to} \quad x \in A
\end{equation}
with $c = c_+$. If $\Gamma(x_n(c_+), x_n, V_n) = 0$ (i.e. $x_n(c_+)$ is feasible for the linearised problem
\eqref{prob:LinearizedProblem}), go to \textbf{Step 4}.
}

\noindent\textbf{Step 2.} {Compute a solution $\widehat{x}_n$ of the optimal feasibility subproblem
\begin{equation} \label{prob:StPen_FeasibilitySubproblem}
  \min_x \: \Gamma(x, x_n, V_n) \quad \text{subject to} \quad x \in A.
\end{equation}
If $\Gamma(\widehat{x}_n, x_n, V_n) < \Gamma(x_n, x_n, V_n)$, go to \textbf{Step 3}. Otherwise,
$x_n$ is a critical point of the penalty term $\varphi$. \textbf{While} the inequality
\[
  \Gamma(x_n(c_+), x_n, V_n) \le \Gamma(\widehat{x}_n, x_n, V_n) + \varepsilon_{feas}
\]
is \textbf{not} satisfied, increase $c_+$ and compute a solution $x_n(c_+)$ of the problem
\eqref{prob:StPen_PenaltySubproblem} with $c = c_+$. Once the inequality is satisfied, go to \textbf{Step 4}.
}

\noindent\textbf{Step 3.} {\textbf{While} the inequality
\begin{equation} \label{eq:FeasibilityRecuction}
  \Gamma(x_n(c_+), x_n, V_n) - \Gamma(x_n, x_n, V_n) 
  \le \eta_1 \Big[ \Gamma(\widehat{x}_n, x_n, V_n) - \Gamma(x_n, x_n, V_n) \Big].
\end{equation}
is \textbf{not} satisfied, increase $c_+$ and compute a solution $x_n(c_+)$ of the problem
\eqref{prob:StPen_PenaltySubproblem} with $c = c_+$. Once inequality \eqref{eq:FeasibilityRecuction} is satisfied, 
go to \textbf{Step 4}.
}

\noindent\textbf{Step 4.} {Put $c_{n + 1} = c_+$. \textbf{While} the condition
\begin{equation} \label{eq:FeasReduction_PenReduction_Updated}
\begin{split}
  Q_{c_{n + 1}}(x_n(c_{n + 1}), x_n, V_n) &- Q_{c_{n + 1}}(x_n, x_n, V_n) 
  \\
  &\le \eta_2 c_{n + 1} \Big[ \Gamma(x_n(c_{n + 1}), x_n, V_n) - \Gamma(x_n, x_n, V_n) \Big]
\end{split}
\end{equation} 
is \textbf{not} satisfied, increase $c_{n + 1}$ and compute a solution $x_n(c_{n + 1})$ of the problem 
\eqref{prob:StPen_PenaltySubproblem} with $c = c_{n + 1}$. Once inequality \eqref{eq:FeasReduction_PenReduction_Updated}
is satisfied, put $x_{n + 1} = x_n(c_{n + 1})$. If a \textbf{stopping criterion} is not satisfied, set $n = n + 1$ and
go to \textbf{Step~1}.
}
\end{algorithm}

\subsection{A discussion of the method}

Let us comment on the steering exact penalty DCA. Firstly, note that one might need to solve the penalty subproblem
\eqref{prob:StPen_PenaltySubproblem} multiple times with increasing values of the penalty parameter in order to find
$c_+$ on Steps 2 and 3, and $c_{n + 1}$ on Step 4. These updated values of the penalty parameter can be computed by
increasing the current value of the penalty parameter by a constant factor $\rho > 1$ (say, $\rho = 10$) and recomputing
a solution $x_n(c_+)$ of the penalty subproblem \eqref{prob:StPen_PenaltySubproblem} for the increased value of the
penalty parameter. If new solution $x_n(c_+)$ satisfies the required condition, then one proceeds to the next step.
Otherwise, the value of the penalty parameter is increased again, till the conditions are satisfied. Below we will show
that under some natural assumptions one can always find $c_+ \ge c_n$ (and $c_{n + 1} \ge c_+$) satisfying these
conditions, which implies that on each iteration of Algorithmic Pattern~\ref{alg:SteeringPenalty} the penalty subproblem
\eqref{prob:StPen_PenaltySubproblem} is solved only a finite number of times.

In the best case, the penalty subproblem \eqref{prob:StPen_PenaltySubproblem} is solved only once per iteration. If
a solution of this problem is feasible for the linearised problem \eqref{prob:LinearizedProblem} and satisfies 
inequality \eqref{eq:FeasReduction_PenReduction_Updated} with $c_{n + 1} = c_n$, then the value of the penalty parameter
$c_n$ is adequate, one sets $x_{n + 1} = x_n(c_n)$ and moves to the next iteration. Note that inequality
\eqref{eq:FeasReduction_PenReduction_Updated} is satisfied automatically, e.g. if $x_n$ is feasible for the original
problem $(\mathcal{P})$ and $x_n(c_n)$ is feasible for the linearised problem \eqref{prob:LinearizedProblem} (in this
case the right-hand side of this inequality is nonnegative, while the left-hand side if nonpositive by the definition of
$x_n(\cdot)$). Let us also note that $x_n(c)$ is feasible for the linearised problem \eqref{prob:LinearizedProblem},
provided the penalty function $Q_c(\cdot, x_n, V_n)$ is \textit{exact} for the linearised problem. Sufficient conditions
for the exactness of this penalty functions in the case of inequality constrained and more general cone constrained DC
optimisation problems can be found in \cite[Sect.~4.5]{Dolgopolik_DCcone}.

However, in the general case on every iteration of the method one has to solve the optimal feasibility problem
\eqref{prob:StPen_FeasibilitySubproblem} once and solve the penalty subproblem \eqref{prob:StPen_PenaltySubproblem} with
increasing values of the penalty parameter multiple times in order to find an adequate value of the penalty parameter.
Benefits of this approach in the context of various SQP-type optimisation method were discussed in details and
illustrated by multiple numerical examples in \cite{ByrdNocedalWaltz,ByrdLopezCalvaNocedal}.

Let us note that before computing $\widehat{x}_n$ on Step~2 it is recommended to first compute the value 
$\Gamma(x_n, x_n, V_n)$. If $\Gamma(x_n, x_n, V_n) = 0$, i.e. if $x_n$ is feasible for the problem $(\mathcal{P})$,
then the optimal value of problem \eqref{prob:StPen_FeasibilitySubproblem} is zero and one can define 
$\widehat{x}_n = x_n$. This way, in some cases one can save time by not solving problem
\eqref{prob:StPen_FeasibilitySubproblem}.

\begin{remark}
As we will show below (see Theorem~\ref{thrm:Correctness}), all steps of Algorithmic Pattern~\ref{alg:SteeringPenalty}
are correctly defined and, at least in theory, the required values of the penalty parameter $c_+$ and $c_{n + 1}$ can
always be found. Nevertheless, for a practical implementation of this algorithmic pattern it seems advisable to replace
the inequality $\Gamma(\widehat{x}_n, x_n, V_n) < \Gamma(x_n, x_n, V_n)$ on Step~2 with the inequality
\[
  \Gamma(\widehat{x}_n, x_n, V_n) < \Gamma(x_n, x_n, V_n) - \varepsilon
\] 
for some small $\varepsilon > 0$. Such replacement might help one to avoid an unnecessary increase of the penalty
parameter caused, in particular, by computational errors.
\end{remark}

\begin{remark}
From the convergence analysis of the method presented below (see Theorem~\ref{thrm:SteeringExPen_Criticality}) it
follows that one can use the following inequalities
\[
  \big| \Phi_{c_n}(x_{n + 1}) - \Phi_{c_n}(x_n) \big| < \varepsilon_f \quad 
  \Big( \text{and/or } \| x_{n + 1} - x_n \| < \varepsilon_x \Big), \quad 
  \varphi(x_{n + 1}) < \varepsilon_{\varphi}
\]
with some prespecified $\varepsilon_f > 0$, $\varepsilon_x > 0$, and $\varepsilon_{\varphi} > 0$ as a stopping criterion
for Algorithmic Pattern~\ref{alg:SteeringPenalty}. According to this criterion one terminates the algorithm, if 
the decrease of the value of the penalty function $\Phi_{c_n}$ is sufficiently small (and/or the difference between two
successive iterates is sufficiently small), and the current iterate $x_{n + 1}$ satisfies the constraints with
prespecified tolerance $\varepsilon_{\varphi}$. Note, however, that an undesirable situation when the method gets stuck
at an infeasible generalised critical point (or an infeasible critical point of the penalty term $\varphi$) is possible.
Therefore, a practical implementation of Algorithmic Pattern~\ref{alg:SteeringPenalty} must contain a suitable
safeguard, which would ensure that the method works properly in the described situation. Namely, if for some 
$n \in \mathbb{N}$ one has
\[
  f_0(x_{n + k + 1}) \approx f_0(x_{n + k}), \quad \varphi(x_{n + k + 1}) \approx \varphi(x_{n + k}) \quad
  \forall k \in \{ 0, 1, \ldots, s \}
\] 
for some fixed $s \in \mathbb{N}$, but the inequality $\varphi(x_{n + s + 1}) < \varepsilon_{\varphi}$ is not satisfied,
then one should either restart the algorithm with a different initial guess $x_0$ or try computing different 
subgradients of the functions $g_k$ and $h_j$, if these functions are nonsmooth at the last computed point.
\end{remark}

\begin{remark} \label{rmrk:StrekalovskyMethod}
A method for solving nonsmooth DC optimisation problems with DC equality and inequality constraints similar to
Algorithmic Pattern~\ref{alg:SteeringPenalty} was studied in the recent paper \cite{Strekalovsky2020}. There are
multiple small technical differences between these methods (e.g. the fact that the $\ell_{\infty}$ penalty term is
used in \cite{Strekalovsky2020} for inequality constraints), which we do not discuss here for the sake of shortness. By
far the main difference between them consists in the fact that the case when 
$\Gamma(\widehat{x}_n, x_n, V_n) = \Gamma(x_n, x_n, V_n)$ (i.e. $x_n$ is a critical point of the penalty term $\varphi$;
see Def.~\ref{def:PenTermCriticality} below) is left out of consideration in the method from \cite{Strekalovsky2020},
which might lead to an incorrect behaviour of this method in some cases and makes the convergence analysis presented in
\cite{Strekalovsky2020} incorrect as well. In particular, there might not exist $\sigma_+ > \sigma_k$ satisfying
\cite[inequality~(7.15)]{Strekalovsky2020}. For example, such $\sigma_+$ does not exist for the problem
\[
  \min \enspace (x_1 - 1)^2 + (x_2 - 1)^2 \quad \text{subject to} \quad x_1^2 - x_2^2 = 0,
\]
if $x^k = (0, 0)^T$ in \cite[Algorithm~7.1]{Strekalovsky2020}.

Let us also point out that the convergence analysis from \cite{Strekalovsky2020} contains several mistakes, which lead
to some erroneous conclusions about the method. In particular, the equality 
$\lim_{k \to \infty} \| x^k - x^{k + 1} \| = 0$ (see \cite[Prp.~7.2]{Strekalovsky2020}) does not imply that $\{ x^k \}$
is a Cauchy sequence and this sequence is convergent, as the example of the sequence $x^k = \sum_{r = 1}^k 1/r$
demonstrates. The convergence of the dual variables $\{ y^k \}$ (see \cite[formula~(7.39)]{Strekalovsky2020}) does not
follow from \cite[Prp.~7.3]{Strekalovsky2020} for the same reason.
\end{remark}

\subsection{Correctness of the method}

In this section we analyse correctness of Algorithmic Pattern~\ref{alg:SteeringPenalty}. In particular, we show that
the conditions on Steps 2--4 of this algorithmic pattern are satisfied for any sufficiently large value of the penalty
parameter, i.e. the subproblems of finding the required values of $c_+$ on Steps 2 and 3 and $c_{n + 1}$ on Step 4 are
always solvable. 

Note at first that if the set $A$ is unbounded, then the function $Q_c(\cdot, x_n, V_n)$ might not attain a global
minimum on the set $A$ for some $n \in \mathbb{N}$, that is, the point $x_n(c_+)$ on Step~1 of Algorithmic
Pattern~\ref{alg:SteeringPenalty} might be undefined. Therefore, hereinafter we suppose that the following assumption
holds true.

\begin{assumption}
For all $n \in \mathbb{N}$ problem \eqref{prob:StPen_PenaltySubproblem} with any $c \ge c_n$ and problem
\eqref{prob:StPen_FeasibilitySubproblem} have optimal solutions, so that all auxiliary optimisation subproblems in
Algorithmic Pattern~\ref{alg:SteeringPenalty} are correctly defined.
\end{assumption}

This assumption is obviously satisfied, when the set $A$ is bounded. Let us provide simple sufficient conditions for 
the validity of this assumption in the case when the set $A$ is unbounded. Recall that a function
$F \colon \mathbb{R}^d \to \mathbb{R}$ is called \textit{coercive} on the set $A$, if $F(x_n) \to + \infty$ as 
$n \to \infty$ for any sequence $\{ x_n \} \subset A$ such that $\| x_n \| \to + \infty$ as $n \to \infty$.

\begin{proposition} \label{prp:WellPosedness}
The following statements hold true:
\begin{enumerate}
\item{if the penalty term 
$\varphi(\cdot) = \sum_{i \in \mathcal{I}} \max\{ f_i(\cdot), 0 \} + \sum_{j \in \mathcal{E}} |f_j(\cdot)|$ is coercive
on $A$, then for any $n \in \mathbb{N}$ problem \eqref{prob:StPen_FeasibilitySubproblem} has globally optimal solutions;
}

\item{if the penalty function $\Phi_{c_0}$ is coercive on $A$, then for any $n \in \mathbb{N}$ and $c \ge c_n$  
the penalty subproblem \eqref{prob:StPen_PenaltySubproblem} has globally optimal solutions.
}
\end{enumerate}
\end{proposition}

\begin{proof}
From the definition of $\Gamma(\cdot, x_n, V_n)$ it follows that $\Gamma(x, x_n, V_n) \ge \varphi(x)$ for all 
$x \in \mathbb{R}^d$ (see~\eqref{def:LinInfeasMeasure} and \eqref{eq:ell1_PenaltyTerm}). Therefore, the function
$\Gamma(\cdot, x_n, V_n)$ is coercive on $A$, which implies that it attains a global minimum on this set, i.e. the point
$\widehat{x}_n$ on Step~2 of Algorithmic Pattern~\ref{alg:SteeringPenalty} is correctly defined for all 
$n \in \mathbb{N}$. The validity of the second statement of the proposition is proved by applying inequalities
\eqref{eq:ConvexMajorant_Steering} and arguing in precisely the same way.
\end{proof}

Next we prove a useful auxiliary result stating, in particular, that the function $c \mapsto \Gamma(x_n(c), x_n, V_n)$
is monotone. This result is important for implementation of Algorithmic Pattern~\ref{alg:SteeringPenalty}, since it
implies that an increase of the penalty parameter $c$ results in a decrease of the infeasibility measure 
$\Gamma(x_n(c), x_n, V_n)$.

\begin{lemma} \label{lem:PenTerm_Obj_Monotone}
For any $n \in \mathbb{N}$ the following statements hold true:
\begin{enumerate}
\item{the function $c \mapsto \Gamma(x_n(c), x_n, V_n)$ is non-increasing;}

\item{$\Gamma(x_n(c), x_n, V_n) \to \Gamma(\widehat{x}_n, x_n, V_n)$ as $c \to + \infty$;}

\item{the function $c \mapsto g_0(x_n(c)) - \langle v_{n0}, x_n(c) \rangle$ is non-decreasing.}
\end{enumerate}
\end{lemma}

\begin{proof}
Fix any $t > c > 0$. By definition $Q_c(x_n(c), x_n, V_n) \le Q_c(x_n(t), x_n, V_n)$ (recall that $x_n(c)$ is a point
of global minimum of the function $Q_c(\cdot, x_n, V_n)$ on the set $A$), which implies that
\begin{multline} \label{eq:PenTermDecrease1}
  g_0(x_n(c)) - g_0(x_n(t)) - \langle v_{n0}, x_n(c) - x_n(t) \rangle 
  \\
  \le  c \Big[ \Gamma(x_n(t), x_n, V_n) - \Gamma(x_n(c), x_n, V_n) \Big].
\end{multline}
Similarly, from the inequality $Q_t(x_n(t), x_n, V_n) \le Q_t(x_n(c), x_n, V_n)$ it follows that
\begin{multline} \label{eq:PenTermDecrease2}
  g_0(x_n(t)) - g_0(x_n(c)) - \langle v_{n0}, x_n(t) - x_n(c) \rangle 
  \\
  \le t \Big[ \Gamma(x_n(c), x_n, V_n) - \Gamma(x_n(t), x_n, V_n) \Big].
\end{multline}
Summing up these two inequalities one obtains that
\[
  (t - c) \Big[ \Gamma(x_n(c), x_n, V_n) - \Gamma(x_n(t), x_n, V_n) \Big] \ge 0,
\]
which yields $\Gamma(x_n(c), x_n, V_n) \ge \Gamma(x_n(t), x_n, V_n)$ due to the fact that $t > c$. Thus, the function 
$c \mapsto \Gamma(x_n(c), x_n, V_n)$ is non-increasing. Hence with the use of \eqref{eq:PenTermDecrease1} one gets
\[
  g_0(x_n(c)) - g_0(x_n(t)) - \langle v_{n0}, x_n(c) - x_n(t) \rangle \le 0,
\]
i.e. the last statement of the lemma hold true.

Let us finally check that $\Gamma(x_n(c), x_n, V_n) \to \Gamma(\widehat{x}_n, x_n, V_n)$ as $c \to + \infty$. Arguing
by reductio ad absurdum, suppose that this statement is false. Then taking into account the definition of
$\widehat{x}_n$ one gets that there exist $\varepsilon > 0$ and an increasing unbounded sequence 
$\{ t_s \} \subset (0, + \infty)$ such that 
\begin{equation} \label{eq:InfeasMeasureNonLimit}
  \Gamma(x_n(t_s), x_n, V_n) \ge \Gamma(\widehat{x}_n, x_n, V_n) + \varepsilon
  \quad \forall s \in \mathbb{N}.
\end{equation}
Observe that by the definition of $x_n(\cdot)$ one has
\[
  Q_{t_s}(\widehat{x}_n, x_n, V_n) \ge Q_{t_s}(x_n(t_s), x_n, V_n) \quad \forall s \in \mathbb{N}.
\]
Hence with the use of inequality \eqref{eq:InfeasMeasureNonLimit} one obtains that
\[
  g_0(\widehat{x}_n) - \langle v_{n0}, \widehat{x}_n - x_n \rangle
  \ge g_0(x_n(t_s)) - \langle v_{n0}, x_n(t_s) - x_n \rangle + t_s \varepsilon \quad \forall s \in \mathbb{N}.
\]
Applying the third statement of the lemma one finally gets that
\[
  g_0(\widehat{x}_n) - \langle v_{n0}, \widehat{x}_n - x_n \rangle
  \ge g_0(x_n(t_1)) - \langle v_{n0}, x_n(t_1) - x_n \rangle + t_s \varepsilon \quad \forall s \in \mathbb{N},
\]
which is impossible, since $t_s \to + \infty$ as $s \to \infty$. Thus, 
$\Gamma(x_n(c), x_n, V_n) \to \Gamma(\widehat{x}_n, x_n, V_n)$ as $c \to + \infty$, and the proof is complete.
\end{proof}

The following theorem states that the subproblems of finding the required values of the penalty parameters $c_+$ and
$c_{n + 1}$ on  Steps 2--4 of Algorithmic Pattern~\ref{alg:SteeringPenalty} are always solvable and, therefore, this
algorithmic pattern is correctly defined.

\begin{theorem} \label{thrm:Correctness}
For any $n \in \mathbb{N}$ the following statements hold true:
\begin{enumerate}
\item{If $\Gamma(\widehat{x}_n, x_n, V_n) = \Gamma(x_n, x_n, V_n)$, then for any $\varepsilon_{feas} > 0$ there exists
$c_* \ge c_n$ such that for all $c_+ \ge c_*$ the inequality 
$\Gamma(x_n(c_+), x_n, V_n) \le \Gamma(\widehat{x}_n, x_n, V_n) + \varepsilon_{feas}$ holds true.
}

\item{If $\Gamma(\widehat{x}_n, x_n, V_n) < \Gamma(x_n, x_n, V_n)$, then for any $\eta_1 \in (0, 1)$ there exists 
$c_* \ge c_n$ such that for all $c_+ \ge c_*$ inequality \eqref{eq:FeasibilityRecuction} holds true.
}

\item{For any $\eta_2 \in (0, 1)$ there exists $c_* \ge c_+$ (here $c_+$ is from Step 4 of Algorithmic
Pattern~\ref{alg:SteeringPenalty}) such that inequality \eqref{eq:FeasReduction_PenReduction_Updated} is satisfied for
all $c_{n + 1} \ge c_*$.
}
\end{enumerate}
\end{theorem}

\begin{proof}
The validity of the first two statement of the theorem follows directly from the fact that 
$\Gamma(x_n(c), x_n, V_n) \to \Gamma(\widehat{x}_n, x_n, V_n)$ as $c \to + \infty$ by
Lemma~\ref{lem:PenTerm_Obj_Monotone}. Let us prove the last statement of the theorem.

Observe that inequality \eqref{eq:FeasReduction_PenReduction_Updated} is satisfied for all $c_{n + 1} \ge c_+$, if
$\Gamma(x_n, x_n, V_n) = 0$ or $\Gamma(x_n(c), x_n, V_n) \ge \Gamma(x_n, x_n, V_n)$ for all $c \ge c_+$, since in this
case the right-hand side of inequality \eqref{eq:FeasReduction_PenReduction_Updated} is nonnegative, while the
left-hand side is nonpositive by the definition of $x_n(\cdot)$.

Thus, one can suppose that $\Gamma(x_n, x_n, V_n) > 0$ and $\Gamma(x_n(t_0), x_n, V_n) < \Gamma(x_n, x_n, V_n)$ for
some $t_0 \ge c_+$. From the first statement of Lemma~\ref{lem:PenTerm_Obj_Monotone} it follows that there exists
$\varepsilon > 0$ such that 
\begin{equation} \label{eq:PenTermDifference}
  \Gamma(x_n(c), x_n, V_n) \le \Gamma(x_n(t_0), x_n, V_n) 
  < \Gamma(x_n, x_n, V_n) - \varepsilon  \quad \forall c \ge t_0.
\end{equation}
Denote $\omega(x) = g_0(x) - \langle v_{n0}, x - x_n \rangle$ for all $x \in \mathbb{R}^d$. By the definition of
$x_n(\cdot)$ one has $Q_c(x_n(c), x_n, V_n) \le Q_c(\widehat{x}_n, x_n, V_n)$ for all $c > 0$, which implies that
\[
  \omega(x_n(c)) \le \omega(\widehat{x}_n) 
  + c \Big( \Gamma(\widehat{x}_n, x_n, V_n) - \Gamma(x_n(c), x_n, V_n) \Big) \le \omega(\widehat{x}_n)
  \quad \forall c > 0,
\]
where the last inequality follows from the fact that $\Gamma(\widehat{x}_n, x_n, V_n) \le \Gamma(x_n(c), x_n, V_n)$ by
definition (see Step~2 of Algorithmic Pattern~\ref{alg:SteeringPenalty}). Consequently, applying inequality
\eqref{eq:PenTermDifference} one obtains that for all 
$c \ge c_* := \max\{ t_0, (\omega(\widehat{x}_n) - \omega(x_n)) / (1 - \eta_2) \varepsilon \}$ the following inequality
holds true
\begin{align*}
  \omega(x_n(c)) - \omega(x_n) &\le \omega(\widehat{x}_n) - \omega(x_n) \le c (1 - \eta_2) \varepsilon 
  \\
  &\le c (\eta_2 - 1) \Big( \Gamma(x_n(c), x_n, V_n) - \Gamma(x_n, x_n, V_n) \Big)
\end{align*}
(recall that $\eta_2 < 1$ by definition). Adding the term $c( \Gamma(x_n(c), x_n, V_n) - \Gamma(x_n, x_n, V_n))$ to both
sides of this inequality one gets
\[
  Q_c(x_n(c), x_n, V_n) - Q_c(x_n, x_n, V_n) 
  \le c \eta_2 \Big( \Gamma(x_n(c), x_n, V_n) - \Gamma(x_n, x_n, V_n) \Big)
\]
for all $c \ge c_*$, that is, inequality \eqref{eq:FeasReduction_PenReduction_Updated} is satisfied for all $c \ge c_*$.
\end{proof}

\subsection{Convergence of the infeasibility measure}

Now we turn to a convergence analysis of Algorithmic Pattern~\ref{alg:SteeringPenalty}. Recall that this algorithmic
pattern is a DCA-type method for minimising the penalty function $\Phi_c(x) = f_0(x) + c \varphi(x)$ for the problem
$(\mathcal{P})$ (see \eqref{eq:ell1_PenaltyTerm}). 

\begin{definition} \label{def:PenTermCriticality}
One says that a point $x_* \in A$ is a \textit{critical} point of the penalty term $\varphi$, if there exist 
$v_k \in \partial h_k(x_*)$, $k \in \mathcal{I} \cup \mathcal{E} \cup \{ 0 \}$, and $w_j \in \partial g_j(x_*)$, 
$j \in \mathcal{E}$, such that $x_*$ is an optimal solution of the problem
\[
  \minimise_x \enspace \Gamma(x, x_*, V) \quad \text{subject to} \quad x \in A,
\]
where $V = (v_0, v_1, \ldots, v_m, w_{\ell + 1}, \ldots, w_m)$.
\end{definition}

Note that by definition $\Gamma(x, x_*, V_*) \ge \varphi(x)$ for all $x \in \mathbb{R}^d$ and the equality
$\Gamma(x_*, x_*, V_*) = \varphi(x_*)$ holds true (see~\eqref{eq:ell1_PenaltyTerm} and \eqref{def:LinInfeasMeasure}),
i.e. the function $\Gamma(\cdot, x_*, V_*)$ is a global convex majorant of the penalty term $\varphi$. With the use of
this fact one can easily check that any feasible point of the problem $(\mathcal{P})$ is critical for the penalty term
$\varphi$, since any such point is a global minimiser of $\varphi$ on $A$. For an infeasible point, the criticality of
the penalty term means that the constraints of the problem $(\mathcal{P})$ are in some sense degenerate at this point. 

Observe also that the inequality $\Gamma(\widehat{x}_n, x_n, V_n) < \Gamma(x_n, x_n, V_n)$ on Step~2 of Algorithmic
Pattern~\ref{alg:SteeringPenalty} simply means that $x_n$ is \textit{not} a critical point of the penalty term
$\varphi$. In particular, Step~3 of the algorithmic pattern is executed, provided $x_n(c_n)$ is infeasible for the
linearised problem \eqref{prob:LinearizedProblem} and $x_n$ is not a critical point of the penalty term $\varphi$.

Let us establish feasibility/infeasibility properties of limit points of sequences generated by Algorithmic
Pattern~\ref{alg:SteeringPenalty}. Namely, the following two theorems contain sufficient conditions for limit points
of the sequence generated by this algorithmic pattern to be either a feasible point of the problem $(\mathcal{P})$ or
an infeasible critical point of the penalty term $\varphi$. Let us note that the assumption on the sequence $\{ x_n \}$
of the second theorem might seem unusual at the first glance. We discuss it in details below, since the proof of
the theorem reveals the essence of this assumption.

\begin{theorem} \label{thrm:SteeringPenalty_Feasibility_1}
Suppose that the sequence $\{ x_n \}$ generated by Algorithmic Pattern~\ref{alg:SteeringPenalty} converges to some
point $x_*$. Then $x_*$ is either a feasible point of the problem $(\mathcal{P})$ or an infeasible critical point of
the penalty term $\varphi$.
\end{theorem}

\begin{proof}
Let the sequence $\{ x_n \}$ converge to some point $x_*$. Arguing by reductio ad absurdum, suppose that $x_*$ is
infeasible for the problem $(\mathcal{P})$, but is not critical for the penalty term $\varphi$. Let us consider two
cases.

\textbf{Case I.}~Suppose at first that there exists a subsequence $\{ x_{n_k} \}$ such that 
\[
  \Gamma(x_{n_k}, x_{n_k}, V_{n_k}) = \Gamma(\widehat{x}_{n_k}, x_{n_k}, V_{n_k}) \quad \forall k \in \mathbb{N},
\]
i.e. each point $x_{n_k}$ is critical for the penalty term $\varphi$. Then by the definition of $\widehat{x}_n$ for any
$k \in \mathbb{N}$ one has
\begin{equation} \label{eq:CriticalSubsequence}
  \Gamma(x_{n_k}, x_{n_k}, V_{n_k}) \le \Gamma(x, x_{n_k}, V_{n_k}) \quad \forall x \in A.
\end{equation}
The sequence of subgradients $\{ V_{n_k} \}$ is bounded by \cite[Thm.~24.7]{Rockafellar} due to the fact that the
sequence $\{ x_n \}$ is bounded as a convergent sequence. Therefore, replacing, if necessary, this sequence with a
subsequence, one can suppose that $\{ V_{n_k} \}$ converges to some 
$V^* = (v_0^*, v_1^*, \ldots, v_{\ell}^*, w_{\ell + 1}^*, \ldots, w_m^*)$. From the fact that the subdifferential
mapping of a convex function is closed (see, e.g. \cite[Thm.~24.4]{Rockafellar}) it follows that 
$v_k^* \in \partial h_k(x_*)$ for all $k \in \{ 0 \} \cup \mathcal{I} \cup \mathcal{E}$ and 
$w_j^* \in \partial g_j(x_*)$ for all $j \in \mathcal{E}$.

Now, passing to the limit in \eqref{eq:CriticalSubsequence} as $k \to \infty$ for each fixed $x \in A$ one gets that
\[
  \Gamma(x_*, x_*, V^*) \le \Gamma(x, x_*, V^*) \quad \forall x \in A.
\]
i.e. $x_*$ is a critical point of $\varphi$, which contradicts our assumption. 

\textbf{Case~II.} Suppose now that there exists $n_0 \in \mathbb{N}$ such that for any $n \ge n_0$ the point $x_n$ is
not critical for $\varphi$. Moreover, one can also assume that $\varphi(x_n) > 0$ for all $n \ge n_0$, since $x_*$ is
infeasible for the problem $(\mathcal{P})$. Therefore, the inequality
\begin{equation} \label{eq:InfeasibilityMeasRelaxation}
  \varphi(x_{n + 1}) < \varphi(x_n) \quad \forall n \ge n_0
\end{equation}
holds true. Indeed, if on Step~1 the equality $\Gamma(x_n(c_n), x_n, V_n) = 0$ holds true, then by
Lemma~\ref{lem:PenTerm_Obj_Monotone} one has
\[
  0 = \Gamma(x_n(c_n), x_n, V_n) \ge \Gamma(x_{n + 1}, x_n, V_n) \ge \varphi(x_{n + 1})
\]
(here we used the fact that $\Gamma(\cdot, x_n, V_n)$ is a global majorant of $\varphi(\cdot)$). Therefore for all
$n \ge n_0$ one has 
\[
  \Gamma(x_n(c_n), x_n, V_n) > 0, \quad \Gamma(x_n, x_n, V_n) > \Gamma(\widehat{x}_n, x_n, V_n),
\]
which implies that for all $n \ge n_0$ Algorithmic Pattern~\ref{alg:SteeringPenalty} executes Step~3 and
\[
  \Gamma(x_n(c_+), x_n, V_n) - \Gamma(x_n, x_n, V_n) 
  \le \eta_1 \Big[ \Gamma(\widehat{x}_n, x_n, V_n) - \Gamma(x_n, x_n, V_n) \Big] < 0
\]
Hence with the use of Lemma~\ref{lem:PenTerm_Obj_Monotone} and the fact that $\Gamma(\cdot, x_n, V_n)$ is a global
majorant of $\varphi(\cdot)$ one obtains that
\[
  \varphi(x_{n + 1}) \le \Gamma(x_{n + 1}, x_n, V_n) \le \Gamma(x_n(c_+), x_n, V_n)
  < \Gamma(x_n, x_n, V_n)  = \varphi(x_n),
\]
that is, inequality \eqref{eq:InfeasibilityMeasRelaxation} holds true.

As was noted above, from the fact that the sequence $\{ x_n \}$ converges to $x_*$ it follows that there exists a
subsequence of subgradients $\{ V_{n_k} \}$ converging to some 
$V^* = (v_0^*, v_1^*, \ldots, v_m^*, w_{\ell + 1}^*, \ldots, w_m^*)$ 
such that $v_k^* \in \partial h_k(x_*)$ for all $k \in \{ 0 \} \cup \mathcal{I} \cup \mathcal{E}$ and 
$w_j^* \in \partial g_j(x_*)$ for all $j \in \mathcal{E}$.

By our assumption $x_*$ is not a critical point of $\varphi$. Therefore there exist $\varepsilon > 0$ and $y \in A$
such that $\Gamma(y, x_*, V_*) \le \Gamma(x_*, x_*, V_*) - \varepsilon$. Consequently, there exists $k_0 \in \mathbb{N}$
such that for all $k \ge k_0$ the following inequalities hold true:
\[
  \Gamma(\widehat{x}_{n_k}, x_{n_k}, V_{n_k}) \le \Gamma(y, x_{n_k}, V_{n_k})
  \le \Gamma(x_{n_k}, x_{n_k}, V_{n_k}) - \frac{\varepsilon}{2}.
\]
Hence taking into account the fact that Algorithmic Pattern~\ref{alg:SteeringPenalty} executes Step~3 for all 
$n \ge n_0$ one gets that
\begin{align*}
  \Gamma(x_{n_k}(c_+), x_{n_k}, V_{n_k}) - \Gamma(x_{n_k}, x_{n_k}, V_{n_k})
  &\le \eta_1 \Big[ \Gamma(\widehat{x}_{n_k}, x_{n_k}, V_{n_k}) - \Gamma(x_{n_k}, x_{n_k}, V_{n_k}) \Big]
  \\
  &\le - \eta_1 \frac{\varepsilon}{2}.
\end{align*}
for any sufficiently large $k$. Therefore, as is easy to check, 
$\varphi(x_{n_k + 1}) \le \varphi(x_{n_k}) - \eta_1 \varepsilon / 2$ for any $k$ large enough, which with the use of
\eqref{eq:InfeasibilityMeasRelaxation} implies that $\varphi(x_n) \to - \infty$ as $n \to \infty$. However, by
definition the function $\varphi$ is nonnegative, which leads to an obvious contradiction.
\end{proof}  

\begin{theorem} \label{thrm:SteeringPenalty_Feasibility_2}
Let $\{ x_n \}$ be the sequence generated by Algorithmic Pattern~\ref{alg:SteeringPenalty}, and suppose that 
$\sum_{n = 0}^{\infty} \max\{ 0, \varphi(x_{n + 1}) - \varphi(x_n) \} < + \infty$. Then limit points of the sequence 
$\{ x_n \}$ are either feasible points of the problem $(\mathcal{P})$ or infeasible critical points of the penalty term
$\varphi$.
\end{theorem}

\begin{proof}
Arguing by reductio ad absurdum, suppose that there exists a limit point $x_*$ of the sequence $\{ x_n \}$ that is
infeasible for the problem $(\mathcal{P})$, but not critical for the  penalty term $\varphi$. By definition there exists
a subsequence $\{ x_{n_k} \}$ converging to $x_*$. The corresponding sequence of subgradients $\{ V_{n_k} \}$ is bounded
by \cite[Thm.~24.7]{Rockafellar}. Therefore, without loss of generality one can suppose that it converges to some
vector $V^* = (v_0^*, v_1^*, \ldots, v_m^*, w_{\ell + 1}^*, \ldots, w_m^*)$. Moreover, one has 
$v_k^* \in \partial h_k(x_*)$ for all $k \in \{ 0 \} \cup \mathcal{I} \cup \mathcal{E}$ and 
$w_j^* \in \partial g_j(x_*)$ for all $j \in \mathcal{E}$ due to the closedness of the subdifferential mapping of a
convex function \cite[Thm.~24.4]{Rockafellar}.

By our assumption $x_*$ is not a critical point of $\varphi$. Therefore there exist $\varepsilon > 0$ and $y \in A$ such
that
\[
  \Gamma(y, x_*, V^*) < \Gamma(x_*, x_*, V^*) - \varepsilon.
\]
Consequently, one can find $k_0 \in \mathbb{N}$ such that
\[
  \Gamma(\widehat{x}_{n_k}, x_{n_k}, V_{n_k}) \le \Gamma(y, x_{n_k}, V_{n_k}) 
  \le \Gamma(x_{n_k}, x_{n_k}, V_{n_k}) - \frac{\varepsilon}{2}
  \quad \forall k \ge k_0.
\]
Therefore, for any such $k$ either $\Gamma(x_{n_k}(c_{n_k}), x_{n_k}, V_{n_k}) = 0$ and Algorithmic
Pattern~\ref{alg:SteeringPenalty} does not execute Steps 2 and 3 or $\Gamma(x_{n_k}(c_{n_k}), x_{n_k}, V_{n_k}) > 0$
and the algorithmic pattern necessarily executes Step~3. In the latter case one has
\begin{align*}
  \Gamma(x_{n_k + 1}, x_{n_k}, V_{n_k}) - \Gamma(x_{n_k}, x_{n_k}, V_{n_k}) 
  &\le \eta_1 \Big[ \Gamma(\widehat{x}_{n_k}, x_{n_k}, V_{n_k}) - \Gamma(x_{n_k}, x_{n_k}, V_{n_k}) \Big] 
  \\
  &\le - \frac{\eta_1 \varepsilon}{2} 
\end{align*}
by Lemma~\ref{lem:PenTerm_Obj_Monotone}, while in the former case this inequality is satisfied due to the fact that
$\Gamma(x_{n_{k + 1}}, x_{n_k}, V_{n_k}) = \Gamma(\widehat{x}_{n_k}, x_{n_k}, V_{n_k}) = 0$, which implies that
\[
  \Gamma(x_{n_k + 1}, x_{n_k}, V_{n_k}) - \Gamma(x_{n_k}, x_{n_k}, V_{n_k}) 
  = \Gamma(\widehat{x}_{n_k}, x_{n_k}, V_{n_k}) - \Gamma(x_{n_k}, x_{n_k}, V_{n_k}) \le - \frac{\varepsilon}{2}.
\]
Now, applying the definitions of $\Gamma$ and $\varphi$ (see~\eqref{def:LinInfeasMeasure} and
\eqref{eq:ell1_PenaltyTerm}) one finally gets that
\[
  \varphi(x_{n_k + 1}) - \varphi(x_{n_k}) \le
  \Gamma(x_{n_k + 1}, x_{n_k}, V_{n_k}) - \Gamma(x_{n_k}, x_{n_k}, V_{n_k}) \le
  - \frac{\eta_1 \varepsilon}{2} \quad \forall k \ge k_0.
\]
By our assumption the sum $\sum_{n = 0}^{\infty} \max\{ 0, \varphi(x_{n + 1}) - \varphi(x_n) \}$ is finite. Therefore,
there exists $k_1 \in \mathbb{N}$ such that 
\[
  \sum_{s = n_k}^{\infty} \max\{ 0, \varphi(x_{s + 1}) - \varphi(x_s) \} 
  \le \frac{\eta_1 \varepsilon}{4} \quad \forall k \ge k_1,
\]
which implies that
\begin{align*}
  \varphi(x_{n_{k + 1}}) - \varphi(x_{n_k}) 
  &= \sum_{s = n_k + 1}^{n_{k + 1} - 1} \Big( \varphi(x_{s + 1}) - \varphi(x_s) \Big)
  + \varphi(x_{n_k + 1}) - \varphi(x_{n_k})
  \\
  &\le \sum_{s = n_k + 1}^{\infty} 
  \max\{ 0, \varphi(x_{s + 1}) - \varphi(x_s) \} + \varphi(x_{n_k + 1}) - \varphi(x_{n_k})
  \le - \frac{\eta_1 \varepsilon}{4}
\end{align*}
for all $k \ge \max\{ k_0, k_1 \}$. Consequently, $\varphi(x_{n_k}) \to - \infty$ as $k \to \infty$, which contradicts
the fact that the function $\varphi$ is nonnegative by definition.
\end{proof}

\begin{remark} \label{rmrk:PenaltyRelaxation}
{(i)~It is worth noting that from the proof of Theorem~\ref{thrm:SteeringPenalty_Feasibility_1} it follows that if
the sequence $\{ x_n \}$ is infeasible for the problem $(\mathcal{P})$ and none of its elements is critical for the
penalty term $\varphi$, then $\varphi(x_{n + 1}) < \varphi(x_n)$ for all $n \in \mathbb{N}$. 
}

\noindent{(ii)~It is easily seen that if on $n$th iteration Algorithmic Pattern~\ref{alg:SteeringPenalty} does not
execute Step~2, then $\varphi(x_{n + 1}) = 0 \le \varphi(x_n)$. Similarly, if the algorithmic pattern executes Step~3,
then $\varphi(x_{n + 1}) < \varphi(x_n)$. However, if $x_n$ is a critical point of the penalty term $\varphi$, then in
the general case only the inequality $\varphi(x_{n + 1}) \le \varphi(x_n) + \varepsilon_{feas}$ holds true. In other
words, the infeasibility measure $\varphi(x_{n + 1})$ might increase, if $x_n$ is an infeasible critical point of the
penalty term. Therefore, the assumption in Theorem~\ref{thrm:SteeringPenalty_Feasibility_2} can be reformulated as
follows: if an infinite number of elements of the sequence $\{ x_n \}$ are critical points of the penalty term 
$\varphi$, then the sum of the increments of the infeasibility measure between iterations 
(i.e. $\max\{ 0, \varphi(x_{n + 1}) - \varphi(x_n) \}$) is finite. Let us note that this assumption is satisfied,
provided only a finite number of points in the sequence $\{ x_n \}$ are critical for the penalty term $\varphi$.
}

\noindent{(iii)~One can ensure that the assumption of Theorem~\ref{thrm:SteeringPenalty_Feasibility_2} is always
satisfied by slightly changing Algorithmic Pattern~\ref{alg:SteeringPenalty}. Namely, it is sufficient to replace the
inequality 
\[
  \Gamma(x_n(c_+), x_n, V_n) < \Gamma(\widehat{x}_n, x_n, V_n) + \varepsilon_{feas}
\]
on Step~2 of Algorithmic Pattern~\ref{alg:SteeringPenalty} with the inequality
\[
  \Gamma(x_n(c_+), x_n, V_n) < \Gamma(\widehat{x}_n, x_n, V_n) + \varepsilon_n
\]
for some sequence $\{ \varepsilon_n \} \subset (0, + \infty)$ such that $\sum_{n = 0}^{\infty} \varepsilon < + \infty$.
However, such replacement is inadvisable for practical implementation of Algorithmic Pattern~\ref{alg:SteeringPenalty},
since it might lead to an unnecessary increase of the penalty parameter $c_n$.

Let us also note that if the penalty function $Q_c(\cdot, x_n, V_n)$ is \textit{exact} for the linearised problem
\eqref{prob:LinearizedProblem}, then one can find $c_+ > 0$ on Step~2 of Algorithmic Pattern~\ref{alg:SteeringPenalty}
such that $\Gamma(x_n(c_+), x_n, V_n) = \Gamma(\widehat{x}_n, x_n, V_n)$. Arguing in a similar way to the proof of
\cite[Lemma~3.4]{ByrdLopezCalvaNocedal}, one can check that this penalty function is always exact for problems with
reverse convex constraints (that is, $g_i \equiv 0$ for all $i \in \mathcal{I}$) and affine equality constraints.
Sufficient conditions for the exactness of this penalty function for inequality constrained DC optimisation problems can
be found in \cite{Dolgopolik_DCcone}. If it is known that $Q_c(\cdot, x_n, V_n)$ is exact, then one can modify Step~2 of
Algorithmic Pattern~\ref{alg:SteeringPenalty} accordingly to guarantee that the assumption of
Theorem~\ref{thrm:SteeringPenalty_Feasibility_2} always holds true.
}
\end{remark}

\subsection{Convergence to critical points}

Next we analyse criticality of limit points of sequences generated by Algorithmic Pattern~\ref{alg:SteeringPenalty}. 
Observe that the sequence of penalty parameters $\{ c_n \}$ from Algorithmic Pattern~\ref{alg:SteeringPenalty} is by
definition nondecreasing. Therefore, either this sequence is bounded above and converges to some $c_* > 0$ or it
increases unboundedly. We divide a convergence analysis of Algorithmic Pattern~\ref{alg:SteeringPenalty} into two parts
corresponding to these two particular types of behaviour of the sequence $\{ c_n \}$. First, we consider the case when
the sequence of penalty parameters is bounded.

To simplify convergence analysis, below we suppose that the penalty parameter $c_+$ on Steps 2 and 3 of Algorithmic
Pattern~\ref{alg:SteeringPenalty} has the form $c_+ = \rho^{s_n} c_n$ for some fixed $\rho > 1$ and some 
$s_n \in \mathbb{N}$ (or $c_+ = c_n + s_n \rho$ for some $\rho > 0$). Similarly, we assume that 
$c_{n + 1} = \rho^{s_n} c_+$ (or $c_{n + 1} = c_+ + s_n \rho$) for some $s_n \in \mathbb{N}$ on Step~4. In other words,
we suppose that the required values of the penalty parameter $c_+$ and $c_{n + 1}$ are found by repeatedly increasing a
current value of the penalty parameter by constant factor $\rho > 1$ (or by constant value $\rho > 0$) till the
corresponding inequality is satisfied. Moreover, we suppose that the penalty parameter is not increased when it is
unnecessary, that is, if the corresponding inequalities are satisfied for $c_+ = c_n$ (or $c_{n + 1} = c_+$), then one
defines $c_+ = c_n$ (or $c_{n + 1} = c_+$) on all iterations of 
Algorithmic Pattern~\ref{alg:SteeringPenalty}.

\begin{theorem} \label{thrm:SteeringExPen_Criticality}
Let $\{ x_n \}$ be the sequence generated by Algorithmic Pattern~\ref{alg:SteeringPenalty}. Suppose that the penalty 
function $\Phi_{c_n}$ is bounded below on the set $A$ for some $n \in \mathbb{N}$ and the corresponding sequence of 
penalty parameters $\{ c_n \}$ is bounded and, thus, converges to some $c_* > 0$. Then the following statements hold
true:
\begin{enumerate}
\item{$|\Phi_{c_n}(x_{n + 1}) - \Phi_{c_n}(x_n)| \to 0$ as $n \to \infty$; \label{st:Convergence1}}

\item{if $h_0$ is strongly convex, then $\| x_{n + 1} - x_n \| \to 0$ as $n \to \infty$, where $\| \cdot \|$ is any
norm on $\mathbb{R}^d$;
\label{st:Convergence2}}
 
\item{all limit points of the sequence $\{ x_n \}$ are generalised critical points for the penalty parameter $c_*$;
\label{st:Convergence3}}

\item{all feasible limit points of the sequence $\{ x_n \}$ are critical for the problem $(\mathcal{P})$.}
\end{enumerate}
\end{theorem}

\begin{proof}
\textbf{Part~\ref{st:Convergence1}.}~From our assumptions it follows that the penalty parameter $c_n$ remains constant
for any sufficiently large $n$, i.e. there exists $n_0 \in \mathbb{N}$ such that $c_n = c_{n_0} = c_*$ for all 
$n \in \mathbb{N}$. Therefore, for any $n \ge n_0$ the point $x_{n + 1}$ is defined as a globally optimal solution of
the problem
\[
  \minimise_x \enspace Q_{c_*}(x, x_n, V_n) \quad \text{subject to} \quad x \in A.
\]
Hence $Q_{c_*}(x_{n + 1}, x_n, V_n) \le Q_{c_*}(x_n, x_n, V_n)$, which with the use of inequalities
\eqref{eq:ConvexMajorant_Steering} implies that $\Phi_{c_*}(x_{n + 1}) \le \Phi_{c_*}(x_n)$ for all $n \ge n_0$, that
is, the sequence $\{ \Phi_{c_*}(x_n) \}_{n \ge n_0}$ is nondecreasing. Moreover, one can conclude that 
$|\Phi_{c_n}(x_{n + 1}) - \Phi_{c_n}(x_n)| \to 0$ as $n \to \infty$, since the penalty function $\Phi_{c_n}$ is bounded
below for some $n \in \mathbb{N}$. 

\textbf{Part~\ref{st:Convergence2}.}~Let $h_0$ be strongly convex. Then there exists $\mu > 0$ such that
\[
  h_0(x_{n + 1}) - h_0(x_n) \ge \langle v_{n0}, x_{n + 1} - x_n \rangle + \frac{\mu}{2} \| x_{n + 1} - x_n \|^2
  \quad \forall n \in \mathbb{N}.
\]
By the definition of $x_{n + 1}$ and inequalities \eqref{eq:ConvexMajorant_Steering} one has
\[
  \Phi_{c_n}(x_n) = Q_{c_n}(x_n, x_n, V_n) - h_0(x_n) \ge Q_{c_n}(x_{n + 1}, x_n, V_n) - h_0(x_n)
  \quad \forall n \in \mathbb{N}.
\]
Hence applying the inequality for $h_0$ and the definitions of subgradients $v_k$ and $w_j$ one obtains that
\[
  \Phi_{c_n}(x_n) \ge Q_{c_n}(x_{n + 1}, x_n, V_n) - h_0(x_n) 
  \ge \Phi_{c_n}(x_{n + 1}) + \frac{\mu}{2} \| x_{n + 1} - x_n \|^2 
  \quad \forall n \in \mathbb{N}
\]
(see \eqref{eq:ConvexMajorant_Steering}). Therefore, by the first statement of the theorem 
$\| x_{n + 1} - x_n \| \to 0$ as $n \to \infty$.

\textbf{Part~\ref{st:Convergence3}.}~Firstly, let us note that the validity of the last statement of the theorem follows
directly from the third statement and Remark~\ref{rmrk:GeneralizedCriticality}. Therefore, it remains to prove that all
limit point of the sequence $\{ x_n \}$ are generalised critical points.

Arguing by reductio ad absurdum, suppose that there exists a limit point $x_*$ of the sequence $\{ x_n \}$ that is not
a generalised critical point for $c_*$. By definition there exists a subsequence $\{ x_{n_k} \}$ converging to $x_*$.
As was noted several times above, the corresponding sequence of subgradients $\{ V_{n_k} \}$ is bounded and without loss
of generality one can suppose that it converges to some 
$V^* = (v_0^*, v_1^*, \ldots, v_m^*, w_{\ell + 1}^*, \ldots, w_m^*)$ with $v_k^* \in \partial h_k(x_*)$ for all 
$k \in \{ 0 \} \cup \mathcal{I} \cup \mathcal{E}$ and $w_j^* \in \partial g_j(x_*)$ for all $j \in \mathcal{E}$.

Since $x_*$ is not a generalised critical point for $c_*$, there exist $y \in A$ and $\varepsilon > 0$ such that
\[
  Q_{c_*}(y, x_*, V^*) < Q_{c_*}(x_*, x_*, V^*) - \varepsilon.
\]
Bearing in mind the facts that $x_{n_k}$ converges to $x_*$, $V_{n_k}$ converges to $V_*$, and $c_n = c_*$ for all 
$n \ge n_0$ one obtains that there exists $k_0 \in \mathbb{N}$ such that 
\[
  Q_{c_*}(x_{n_k + 1}, x_{n_k}, V_{n_k}) \le Q_{c_*}(y, x_{n_k}, V_{n_k}) 
  \le Q_{c_*}(x_{n_k}, x_{n_k}, V_{n_k}) - \frac{\varepsilon}{2}
\]
for all $k \ge k_0$. Applying inequalities \eqref{eq:ConvexMajorant_Steering} one gets that
\begin{align*}
  \Phi_{c_*}(x_{n_k + 1}) &\le Q_{c_*}(x_{n_k + 1}, x_{n_k}, V_{n_k}) - h_0(x_{n_k}) 
  \\
  &\le Q_{c_*}(x_{n_k}, x_{n_k}, V_{n_k}) - h_0(x_{n_k}) - \frac{\varepsilon}{2}
  = \Phi_{c_*}(x_{n_k}) - \frac{\varepsilon}{2}
\end{align*}
for all $k \ge k_0$. Hence with the use of the fact that the sequence $\{ \Phi_{c_*}(x_n) \}_{n \ge n_0}$ is
nondecreasing one can conclude that $\Phi_{c_*}(x_n) \to - \infty$ as $n \to \infty$, which contradicts our assumption
that the function $\Phi_{c_n}$ is bounded below on the set $A$ for some $n \in \mathbb{N}$.
\end{proof}

\begin{remark}
Note that the assumption on the strong convexity of the function $h_0$ is not restrictive, since one can always use 
the following DC decomposition of the objective function: $f_0(x) = ( g_0(x) + \mu |x|^2 ) - ( h_0(x) + \mu |x|^2 )$ for
some $\mu > 0$, where $|\cdot|$ is the Euclidean norm.
\end{remark}

\subsection{Sufficient conditions for the boundedness of the penalty parameter}

Let us now turn to the case when the penalty parameter $c_n$ increases unboundedly. This assumption significantly
complicates convergence analysis of the steering exact penalty DCA, and it is unclear whether limit points of 
the sequence generated by this method are (in some sense) generalised critical points in this case. 

A detailed analysis of the case when the penalty parameter increases unboundedly lies outside the scope of this paper.
Here we present only a partial convergence result of this case under the assumption that there are no equality 
constraints (i.e. $\mathcal{E} = \emptyset$). This result demonstrates that if the penalty parameter increases
unboundedly, but the sequence generated by the method converges to a feasible point of the problem $(\mathcal{P})$, then
a natural constraint qualification fails at the limit point (cf. \cite[Thm~3.12, part~(b)]{ByrdLopezCalvaNocedal} and
\cite[Thm.~7]{Dolgopolik_DCcone}). The following definition describes this constraint qualification.

\begin{definition}
Let $\mathcal{E} = \emptyset$ and $x_*$ be a given point. One says that \textit{linearised Slater's condition} holds
true at $x_*$, if for any $v_i \in \partial h_i(x_*)$, $i \in \mathcal{I}$, there exists $y \in A$ such that 
\[  
  g_i(y) - h_i(x_*) - \langle v_i, y - x_* \rangle < 0 \quad \forall i \in \mathcal{I}.
\]
\end{definition}

Note that in the case when the inequality constraints of the problem $(\mathcal{P})$ are convex (i.e. $h_i \equiv 0$ for
all $i \in \mathcal{I}$), linearised Slater's condition coincides with the classical Slater's condition from convex
programming. In the case of problems with reverse convex constraints (i.e. $g_i \equiv 0$ for all $i \in \mathcal{I}$),
linearised Slater's conditions is satisfied at a feasible point $x_*$, provided 
$0 \notin \co\{ \partial h_i(x_*) \mid i \in \mathcal{I}(x_*) \}$, which can be easily verified with the use of the
separation theorem. One can check that the Mangasarian-Fromivitz constraint qualification also implies the validity of
linearised Slater's condition.

\begin{theorem} \label{thrm:SteeringExPen_SlaterLimit}
Let $\mathcal{E} = \emptyset$, the function $f_0$ be coercive on the feasible set of the problem $(\mathcal{P})$, and
$\{ x_n \}$ be the sequence generated by Algorithmic Pattern~\ref{alg:SteeringPenalty}. Suppose that this sequence
converges to a point $x_*$ at which linearised Slater's condition holds true. Then $x_*$ is feasible and critical for
the problem $(\mathcal{P})$, and the corresponding sequence of penalty parameters $\{ c_n \}$ is bounded.
\end{theorem}

We divide the proof of this theorem into two lemmas. At first, we show that the validity of linearised Slater's
condition at $x_*$ implies that for any sufficiently large $n \in \mathbb{N}$ the penalty function 
$Q_c(\cdot, x_n, V_n)$ is \textit{exact} for the corresponding linearised problem (i.e. its points of global minimum
on the set $A$ are feasible for the linearised problem). Then we will use this auxiliary result to prove the statement
of the theorem.

\begin{lemma} \label{lem:ExactPenalty}
Under the assumptions of Theorem~\ref{thrm:SteeringExPen_SlaterLimit} there exists $c_* > 0$ and a neighbourhood
$\mathcal{U}(x_*)$ of $x_*$ such that for all $c \ge c_*$ and $x \in \mathcal{U}(x_*)$ and for any 
$V = (v_0, \ldots, v_{\ell})$ with $v_k \in \partial h_k(x)$, $k \in \mathcal{I} \cup \{ 0 \}$, globally optimal
solutions $z_c(x, V)$ of the problem
\[
  \minimise_z \enspace Q_c(z, x, V) \quad \text{subject to} \quad z \in A 
\]
are feasible for the corresponding linearised problem, i.e., $\Gamma(z_c(x, V), x, V) = 0$ (in other words, 
the penalty function $Q_c$ for the corresponding linearised problem is \textit{exact}). 
\end{lemma}

\begin{proof}
Denote $B(x, r) = \{ y \in \mathbb{R}^d \mid \| y - x \| \le r \}$. We divide the proof of the lemma into three parts.

\textbf{Part~1. The uniform error bound for $\Gamma(z, x, V)$.} Choose any $v_i \in \partial h_i(x_*)$, 
$i \in \mathcal{I}$. By linearised Slater's condition there exist $y \in A$ and $\varkappa > 0$ such that
\[
  g_i(y) - h_i(x_*) - \langle v_i, y - x_* \rangle \le - \varkappa \quad \forall i \in \mathcal{I}.
\]
Therefore there exists $r > 0$ such that 
\[
  g_i(z) - h_i(x) - \langle w_i, z - x \rangle \le - \frac{\varkappa}{2} \quad \forall i \in \mathcal{I}
\]
for all $x \in B(x_*, r)$, $w_i \in B(v_i, r)$, and $z \in B(y, r)$. Hence with the use of the compactness of
subdifferential one can verify that there exist $\varkappa > 0$ and $r > 0$ such that for all 
$v_i \in B(\partial h_i(x_*), r)$, $i \in \mathcal{I}$, one can find $y \in A$ such that for all $z \in B(y, r)$ and 
$x \in B(x_*, r)$ the following inequalities hold true:
\[
  g_i(z) - h_i(x) - \langle v_i, z - x \rangle \le - \varkappa \quad \forall i \in \mathcal{I}.
\]
Here $B(\partial h_i(x_*), r) = \{ w \in \mathbb{R}^d \mid \dist(w, \partial h_i(x_*)) \le r \}$. 

Applying the outer semicontinuity of subdifferential mapping \cite[Crlr.~24.5.1]{Rockafellar} one can find
$\delta > 0$ such that $\partial h_i(x) \subseteq B(\partial h_i(x_*), r)$ for all $x \in B(x_*, \delta)$, which implies
that for all $x \in B(x_*, \delta)$ and $v_i \in \partial h_i(x)$, $i \in \mathcal{I}$, one can find $y(x, V) \in A$
such that 
\begin{equation} \label{eq:UniformSlater}
  g_i(z) - h_i(x) - \langle v_i, z - x \rangle \le - \varkappa \quad \forall i \in \mathcal{I} 
  \quad \forall z \in B(y(x, V), r).
\end{equation}
For any $x \in A$ and $v_k \in \partial h_k(x)$, $k \in \{ 0 \} \cup \mathcal{I}$, denote by
\[
  \Omega(x, V) = \Big\{ z \in A \Bigm| g_i(z) - h_i(x) - \langle v_i, z - x \rangle \le 0 \Big\}
\]
the feasible region of the corresponding linearised problem. Note that this set is nonempty, if $x \in B(x_*, \delta)$,
by the virtue of \eqref{eq:UniformSlater}. 

Fix any $x \in B(x_*, \delta)$, as well as any $v_k \in \partial h_k(x)$, $k \in \mathcal{I}$, and introduce the
functions $\omega_i(z, x, V) := g_i(z) - h_i(x) - \langle v_i, z - x \rangle$. By the definition of the set 
$\Omega(x, V)$ for any $z \in A \setminus \Omega(x, V)$ there exists $i \in \mathcal{I}$ such that 
$\omega_i(z, x, V) > 0$, which due to \eqref{eq:UniformSlater} implies that $\| z - y(x, V) \| \ge r$. For any such $i$
one has
\[
  \omega_i(\alpha z + (1 - \alpha) y(x, V), x, V) 
  \le \alpha \omega_i(z, x, V) + (1 - \alpha) \omega_i(y(x, V), x, V) \le 0
\]
for all $\alpha > 0$ satisfying the inequality
\[
  \alpha \le \frac{-\omega_i(y(x, V), x, V)}{\omega_i(z, x, V) - \omega_i(y(x, V), x, V)}
\]
(note that $\omega_i(y(x, V), x, V) < 0$ due to \eqref{eq:UniformSlater}, which implies that $\alpha \in (0, 1)$).
Therefore $\alpha z + (1 - \alpha) y(x, V) \in \Omega(x, V)$ for all $\alpha \in (0, \alpha_*]$, where
\[
  \alpha_* = \min\left\{ \frac{-\omega_i(y(x, V), x, V)}{\omega_i(z, x, V) - \omega_i(y(x, V), x, V)} 
  \biggm| i \in \mathcal{I} \colon \omega_i(z, x, V) > 0 \right\} > 0.
\]
Hence applying inequalities \eqref{eq:UniformSlater} and $\| z - y(x, V) \| \ge r$ one obtains that
\begin{multline*}
  \dist\big( z, \Omega(x, V) \big) \le \big\| z - \big[\alpha_* z + (1 - \alpha_*) y(x, V) \big] \big\|
  = \| z - y(x, V) \| (1 - \alpha_*) 
  \\
  \le r \max_{i \in \{ k \in \mathcal{I} \colon \omega_k(z, x, V) > 0 \}} 
  \frac{\omega_i(z, x, V)}{\omega_i(z, x, V) - \omega_i(y(x, V), x, V)}
  \le \frac{r}{\varkappa} \max_{i \in \mathcal{I}} \omega_i(z, x, V).
\end{multline*}
for any $z \in A \setminus \Omega(x, V)$. Consequently, the inequalities
\begin{equation} \label{eq:UniformErrorBound}
\begin{split}
  \Gamma(z, x, V) &= \sum_{i \in \mathcal{I}} \max\big\{ 0, g_i(z) - h_i(x) - \langle v_i, z - x \rangle \big\}
  \\
  &\ge \max_{i \in \mathcal{I}} \omega_i(z, x, V) 
  \ge \frac{\varkappa}{r} \dist\big( z, \Omega(x, V) \big)
\end{split}
\end{equation}
hold true for all $z \in A$, $x \in B(x_*, \delta)$, and $V = (v_0, v_1, \ldots, v_{\ell})$, $v_k \in \partial h_k(x)$
(here we used the fact that for any $z \in \Omega(x, V)$ one has $\Gamma(z, x, V) = 0$).

\textbf{Part~2. The exactness of the penalty function $Q_c(\cdot, x, V)$.} Suppose now that there exists $R > 0$ such
that for any $x \in B(x_*, \delta)$ and $V = (v_0, v_1, \ldots, v_{\ell})$, where $v_k \in \partial h_k(x)$, globally
optimal solutions of the linearised problem
\begin{equation} \label{prob:LinearizedNearSlaterPoint}
\begin{split}
  &\minimise_z \enspace g_0(z) - \langle v_0, z - x \rangle 
  \\
  &\text{subject to} \enspace g_i(z) - h_i(x) - \langle v_i, z - x \rangle \le 0, \quad i \in \mathcal{I}, 
  \quad z \in A
\end{split}
\end{equation}
exist and lie within the ball $B(0, R)$. 

The function $g_0$ is Lipschitz continuous with constant $L > 0$ on $B(0, R + 1)$ by \cite[Thm.~10.4]{Rockafellar},
which implies that the objective function of problem \eqref{prob:LinearizedNearSlaterPoint} is Lipschitz continuous on
$B(0, R + 1)$ with constant $L + K$, where 
\[
  K = \sup\Big\{ \| v_0 \| \Big| v_0 \in \partial h_0(x), \: x \in B(x_*, \delta) \Big\} < + \infty.
\] 
Therefore by \cite[Prp.~2.7]{Dolgopolik_ExactPenalty} for any globally optimal solution $z_*$ of problem
\eqref{prob:LinearizedNearSlaterPoint} there exists $\beta > 0$ such that for all $z \in B(z_*, \beta) \cap A$ one has
\[  
  g_0(z) - \langle v_0, z - x \rangle 
  \ge g_0(z_*) - \langle v_0, z_* - x \rangle - (L + K) \dist(z, \Omega(x, V)).
\]
Hence bearing in mind inequality \eqref{eq:UniformErrorBound} one obtains that
\begin{align*}
  Q_c(z, x, V) &= g_0(z) - \langle v_0, z - x \rangle + c \Gamma(z, x, V) 
  \\
  &\ge g_0(z_*) - \langle v_0, z_* - x \rangle - (L + K) \dist(z, \Omega(x, V)) 
  + c \frac{\varkappa}{r} \dist\big( z, \Omega(x, V) \big) 
  \\
  &\ge g_0(z_*) - \langle v_0, z_* - x \rangle = Q_c(z_*, x, V)
\end{align*}
for all $z \in B(z_*, \beta) \cap A$ and $c \ge c_* := (L + K)r / \varkappa$. Thus, for any such $c$ the point $z_*$
is a local minimiser of the function $Q_c(\cdot, x, V)$ on the set $A$. Since this function is convex, $z_*$ is a point
of global minimum of $Q_c(\cdot, x, V)$ on the set $A$ for any $c \ge c_*$ and for all $x \in B(x_*, \delta)$ and 
$V = (v_0, v_1, \ldots, v_{\ell})$, $v_k \in \partial h_k(x)$. Consequently, if $c > c_*$ and $z(c)$ is a point of
global minimum of $Q_c(\cdot, x, V)$ on the set $A$, then $\Gamma(z(c), x, V) = 0$, since otherwise
\[
  Q_c(z_*, x, V) = Q_c(z(c), x, V) > Q_{c_*}(z(c), x, V) \ge Q_{c_*}(z_*, x, V), 
\]
which is impossible. Thus, we found a neighbourhood $B(x_*, \delta)$ of $x_*$ and $c_* > 0$ such that for any $c > c_*$,
$x \in B(x_*, \delta)$, and $V = (v_0, v_1, \ldots, v_{\ell})$, $v_k \in \partial h_k(x)$, points of global minimum
$z_c(x, V)$ of the function $Q_c(\cdot, x, V)$ on the set $A$ are feasible for problem 
\eqref{prob:LinearizedNearSlaterPoint} or, equivalently, $\Gamma(z_c(x, V), x, V) = 0$, which was precisely our goal.

\textbf{Part~3. The boundedness of globally optimal solutions.} To finish the proof of the lemma, it remains to show
that globally optimal solutions of problem \eqref{prob:LinearizedNearSlaterPoint} exist and lie within some ball.
Indeed, fix any $x \in B(x_*, \delta)$ and $V = (v_0, v_1, \ldots, v_{\ell})$ with $v_k \in \partial h_k(x)$, and
denote by $\Omega$ the feasible region of the problem $(\mathcal{P})$. It is easily seen that 
$\Omega(x, V) \subseteq \Omega$.

Introduce the set
\[
  \mathcal{S}(x, V) = \Big\{ z \in \Omega(x, V) \Bigm| g_0(z) - \langle v_0, z - x \rangle \le g_0(x) \Big\}.
\]
This set is not empty, since $x \in \mathcal{S}(x, V)$. Furthermore, with the use of the inequality
$h_0(y) - h_0(x) \ge \langle v_0, y - x \rangle$ one can readily check that $\mathcal{S}(x, V)$ is contained in 
the set
\[
  \Big\{ y \in \Omega \Bigm| f_0(y) \le f_0(x) \Big\} 
  \subseteq \Big\{ y \in \Omega \Bigm| f_0(y) \le \sup_{x \in B(x_*, \delta)} f_0(x) \Big\}.
\]
The last set is bounded due to our assumption that the function $f_0$ is coercive on the set $A$. Thus, the set
$\mathcal{S}(x, V)$ is compact. Consequently, by the standard compactness argument for any $x \in B(x_*, \delta)$ and
for all $V = (v_0, v_1, \ldots, v_{\ell})$ with $v_k \in \partial h_k(x)$, $k \in \mathcal{I} \cup \{ 0 \}$, globally
optimal solutions of problem \eqref{prob:LinearizedNearSlaterPoint} exist and lie within the bounded set
$\{ y \in \Omega \mid f_0(y) \le \sup_{x \in B(x_*, \delta)} f_0(x) \}$.
\end{proof}

\begin{lemma}
Let the assumptions of Theorem~\ref{thrm:SteeringExPen_SlaterLimit} be valid. Then the point $x_*$ is feasible and
critical for the problem $(\mathcal{P})$, and the corresponding sequence of penalty parameters $\{ c_n \}$ is bounded.
\end{lemma}

\begin{proof}
Arguing by reductio ad absurdum, suppose that the sequence of penalty parameters $\{ c_n \}$ increases
unboundedly. Then there exists $n_0 \in \mathbb{N}$ such that for all $n \ge n_0$ one has $c_n \ge c_*$ and 
$x_n \in \mathcal{U}(x_*)$, where $c_* > 0$ and $\mathcal{U}(x_*)$ are from Lemma~\ref{lem:ExactPenalty}.
Therefore, by Lemma~\ref{lem:ExactPenalty} one has $\Gamma(x_n(c_n), x_n, V_n) = 0$ for all $n \ge n_0$, which implies
that Algorithmic Pattern~\ref{alg:SteeringPenalty} does not execute Steps 2 and 3 for all $n \ge n_0$. Moreover,
$\Gamma(x_{n + 1}, x_n, V_n) = 0$ for all $n \ge n_0$ by virtue of Lemma~\ref{lem:PenTerm_Obj_Monotone} and the fact
that $c_{n + 1} \ge c_n$. Consequently, for any $n \ge n_0 + 1$ one has
\[
  Q_{c_n}(x_n(c_n), x_n, V_n) - Q_{c_n}(x_n, x_n, V_n) \le 0
  = c_n \eta_2 \Big[ \Gamma(x_n(c_n), x_n, V_n) - \Gamma(x_n, x_n, V_n) \Big],
\]
i.e. the inequality on Step~4 of Algorithmic Pattern~\ref{alg:SteeringPenalty} is satisfied for $c_{n + 1} = c_n$. Hence
$c_{n + 1} = c_n$ for all $n \ge n_0 + 1$, which contradicts our assumption. Thus, the sequence $\{ c_n \}$ is bounded.

It remains to show that the point $x_*$ is feasible for the problem $(\mathcal{P})$. Then applying
Theorem~\ref{thrm:SteeringExPen_Criticality} one can conclude that the point $x_*$ is critical for this problem. 

Arguing by reductio ad absurdum, suppose that the point $x_*$ is infeasible. Let us show that under this assumption the
point $x_*$ is not critical for the penalty term $\varphi$, which contradicts 
Theorem~\ref{thrm:SteeringPenalty_Feasibility_1}.  

Indeed, fix any $V = (v_0, v_1, \ldots, v_{\ell})$, where $v_k \in \partial h_k(x_*)$, $k \in \mathcal{I} \cup \{ 0 \}$.
Observe that for any $i \in \mathcal{I}$ such that $f_i(x_*) = g_i(x_*) - h_i(x_*) \le 0$ one has
\[
  \max\big\{ 0, g_i(y) - h_i(x_*) - \langle v_i, y - x_* \rangle \big\} = 0
  = \max\big\{ 0, f_i(x_*) \}
\]
(here $y \in A$ is from linearised Slater's condition), while for any $i \in \mathcal{I}$ such that $f_i(x_*) > 0$ (note
that there is at least one such $i$, since $x_*$ is infeasible) one has
\[
  \max\big\{ 0, g_i(y) - h_i(x_*) - \langle v_i, y - x_* \rangle \big\} = 0
  < \max\big\{ 0, f_i(x_*) \}.
\]
Summing up these equalities and inequalities one obtains that
\[
  \Gamma(y, x_*, V) < \Gamma(x_*, x_*, V),
\]
that is, the point $x_*$ is not a point of global minimum of the function $\Gamma(\cdot, x_*, V)$ on the set $A$. Since
the subgradients $V$ were chosen arbitrarily, one can conclude that $x_*$ is not a critical point of the penalty term
$\varphi$ (see~Def.~\ref{def:PenTermCriticality}), and the proof is complete.
\end{proof}

\begin{corollary}
Let $\mathcal{E} = \emptyset$, the function $f_0$ be coercive on the feasible set of the problem $(\mathcal{P})$, and
$\{ x_n \}$ be the sequence generated by Algorithmic Pattern~\ref{alg:SteeringPenalty}. Suppose that this sequence
converges to a point $x_*$ and the corresponding sequence of the penalty parameters $\{ c_n \}$ increases unboundedly.
Then linearised Slater's condition does not hold true at $x_*$. 
\end{corollary}

\section{Numerical experiments}
\label{sect:NumericalExamples}

Let us present two numerical examples illustrating the performance of the steering exact penalty DCA. In both examples,
the initial value of the penalty parameter was chosen as $c_0 = 10$, and the penalty parameter $c_n$ was increased by
the factor $\rho = 10$, each time the corresponding inequality was not satisfied. Parameters of Algorithmic
Pattern~\ref{alg:SteeringPenalty} were chosen as follows: $\eta_1 = \eta_2 = 0.1$ and $\varepsilon_{feas} = 0.01$.

The two following inequalities
\[
  \varphi(x_{n + 1}) < 10^{-3}, \quad \Phi_{c_n}(x_{n + 1}) - \Phi_{c_n}(x_n) < 10^{-3}.
\]
were used as the termination criterion. Finally, all convex optimisation subproblems were solved with the use of
\texttt{cvx}, a {\sc Matlab} package for specifying and solving convex programmes \cite{CVXPackage,GrantBoyd_CVX}.

\begin{example}
Let us consider the nonsmooth discrete-time optimal control problem from \cite[Example~2.1]{Outrata88}, arising as a
simple production model in economics. The problem is formulated as the problem of minimising the cost function
\[
  \enspace \sum_{i = 1}^{k - 1} e^{-\vartheta i} \Big[ - p(i) \min\{ z(i) + u(i), v(i) \} 
  + H(u(i))
  + \beta \max\{ 0, v(i) - z(i) - u(i) \} + \rho z(i) \Big]
\]
subject to 
\begin{align*}
  &z(i + 1) = z(i) + u(i) - \min\{ z(i) + u(i), v(i) \}, \quad z(0) = z_0, 
  \\
  &0 \le u(i) \le b(i) \quad \forall i \in \{ 0, 1, \ldots, k - 1 \}.
\end{align*}
Here $k \in \mathbb{N}$ is a given finite horizon, $z(i) \in [0, + \infty)$ is the amount of products in the stock at
time $i$, $u(i)$ is the production in the time interval $[i, i + 1)$, $v(i)$ is the planned output (supply) at the time
$i$, $p(i)$ is the unit selling price at the time $i$, while the function $H \colon \mathbb{R} \to [0, + \infty)$
represent the production cost. In our numerical experiments we set $H(u) = 0.5 u^2$. The parameter $\vartheta > 0$ is
the interest rate, $\beta > 0$ is the unit penalisation for the case when the amount of available products is less
than the planned output, while $\rho > 0$ is the storage cost.

The problem consists in finding a sequence of ``control inputs'' $u(0), u(1), \ldots, u(k - 1)$ that maximises the total
profit of the factory with respect to the given sequences of outputs $\{ v(i) \}$ and prices $\{ p(i) \}$. The
objective function of this problem is convex, while the nonlinear dynamic constraints can obviously be rewritten as 
DC equality constraints of the form $f_j(z, u) = g_j(z, u) - h_j(z, u) = 0$ with
\[
  g_j(z, u) = 0, \quad h_j(z, u) =  - z(j + 1) + z(j) + u(j) + \max\{ - z(j) - u(j), - v(j) \}.
\] 
The bounds $0 \le u(i) \le b(i)$ on control inputs were included into the nonfunctional constraint $(z, u) \in A$ in our
numerical experiments (i.e. they were not included into the penalty function).

Since no values for parameters of the problem were provided in \cite{Outrata88}, we chose or generated them randomly. 
We used the following values of parameters for our experiments: $\vartheta = 0.01$, $\beta = 5$, $\rho = 0.5$, and 
$z_0 = 0$. The sequences of prices $\{ p(i) \}$, outputs $\{ v(i) \}$, and control bounds $\{ b(i) \}$ were generated
with the use of a random number generator with uniform distribution in the intervals $[10, 15]$, $[5, 15]$, and 
$[5, 15]$, respectively. We also set $k = 1000$, which corresponds to the DC optimisation problem of dimension 
$d = 1999$ with $999$ DC equality constraints (note that $z(0)$ is fixed, while $z(k)$ is redundant for the optimisation
problem).

Numerical experiments showed that the problem under consideration has multiple critical points that are not globally
optimal. Therefore, we solved the problem 10 times using 10 different randomly generated starting points satisfying 
the inequalities $0 \le u(i) \le 15$ and $0 \le z(i) \le 15$. All starting points were infeasible.

In all 10 cases, the algorithm terminated after finding a feasible critical point. The number of iterations before
termination was between 7 and 9, depending on the starting point, while the average computation time was 11 min and
40 s. In all cases the penalty parameter was increased only once to ensure sufficient decay of the infeasibility
measure $\Gamma(x_n(c), x_n, V_n)$ on Step 3 of the penultimate iteration. After this, the method resolved penalised
subproblem \eqref{prob:StPen_PenaltySubproblem} for the increased value of the penalty parameter $c = 100$, whose
solution turned out to be a feasible critical point. The last iteration was needed to check the criticality of this
point. Thus, if $n$ is the number of iterations before termination, the method solved $n + 1$ convex penalised
subproblems \eqref{prob:StPen_PenaltySubproblem}, and $n - 1$ optimal feasibility subproblems
\eqref{prob:StPen_FeasibilitySubproblem} for all randomly generated starting points.
\end{example}

\begin{example}
Let us now consider a slightly simplified version of the nonsmooth optimal control problem from \cite{Outrata83} defined
as follows:
\begin{align*}
  &\minimise_{(u, x, y)} \enspace J(u, x, y) = \int_0^T y(t) \max\{ 0, u(t) \} \, dt
  \\
  &\text{subject to} \quad \dot{x}(t) = y(t),
  \quad
  \dot{y}(t) = \frac{1}{m} u(t) - P y(t) |y(t)| - Q y(t), \quad t \in [0, T],
  \\
  &x(0) = y(0) = 0, \quad x(T) = s, \quad y(T) = 0, 
  \quad
  \underline{u} \le u(t) \le \overline{u} \quad t \in [0, T].
\end{align*}
The problem consists in driving a train from point $0$ to point $s > 0$ and stopping there in time $T > 0$. Here $x(t)$
is the position of the train at time $t$, $y(t)$ is its speed, $m$ is the mass of the train, while parameters $P > 0$
and $Q > 0$ describe its dynamic behaviour. For the sake of simplicity we excluded discontinuous pointwise state
constraints given in \cite{Outrata83}.

To apply the method developed in this paper, the problem was discretised in time. Suppose that the time interval 
$[0, T]$ is divided into $k \in \mathbb{N}$ subintervals of equal length $\Delta = T / k$. Denote
\[
  u(i) = u(\Delta i), \quad x(i) = x(\Delta i), \quad y(i) = y(\Delta i), \quad \forall i \in \{ 0, 1, \ldots, k \}.
\]
Then we obtain the following discretised optimal control problem:
\begin{align*}
  &\minimise_{(u, x, y)} \enspace \sum_{i = 1}^{k - 1} y(i) \max\{ 0, u(i) \}
  \\
  &\text{subject to} \quad x(i + 1) - x(i) = \Delta y(i), 
  \\
  &y(i + 1) - y(i) = \frac{\Delta}{m} u(t) - \Delta P y(i) |y(i)| - \Delta Q y(i), 
  \quad i \in \{ 0, \ldots, k - 1 \}.
  \\
  &x(0) = y(0) = 0, \quad x(k) = s, \quad y(k) = 0, \quad
  \underline{u} \le u(i) \le \overline{u} \quad i \in \{ 0, \ldots, k - 1 \}.
\end{align*}
All linear constraints were included into the nonfunctional convex constraint $(u, x, y) \in A$ (i.e. they were not
included into the penalty function). DC decompositions of the objective function and nonlinear constraints were
constructed with the use of the following equalities
\begin{gather*}
  y [u]_+ = \frac{1}{2} \Big( \big( [y]_+ + [u]_+ \big)^2 + [-y]_+^2 \Big)
  - \frac{1}{2} \Big( \big( [-y]_+ + [u]_+ \big)^2 + [y]_+^2 \Big),
  \\
  y |y| = [y]_+^2 - [- y]_+^2.
\end{gather*}
which can be readily verified directly. Here $[t]_+ = \max\{ 0, t \}$.

In our numerical experiments, we used the same parameters of the problem as given in \cite[Example~5.1]{Outrata83}.
Namely, we set
\[
  \underline{u} = - \frac{2}{3} 10^5, \quad \overline{u} = \frac{2}{3} 10^5, \quad
  P = 0.78 \cdot 10^{-4}, \quad Q = 0.28 \cdot 10^{-3},
\]
and also put $s = 200$ and $T = 48$. Numerical experiments showed that the mass of the train $m = 3 \cdot 10^5$ given
in \cite[Example~5.1]{Outrata83} makes the problem infeasible. Therefore, we set $m = 10^5$ to make sure that the
feasible region of the problem is nonempty. Finally, we chose initial guess $u_0(i) \equiv 0$, $x_0(i) \equiv 0$, and 
$y_0(i) \equiv 0$, and defined $k = 480$, which corresponds to discretisation intervals of length $\Delta = 0.1$
and the DC optimisation problem of dimension $d = 1438$ with $439$ DC equality constraints (note that $x(1) = 0$ and
$y(1) = (\Delta / m) u(0)$, since $y(0) = 0$).

It should be noted that the position of the train $x(i)$ is bounded above by $200$ and the speed of the train,
according to our numerical experiments, lies between $0$ and $6$ even for small values of $k$. In contrast, the control
inputs $u(i)$ (and the objective function) take values proportional to $10^5$. Therefore, to avoid ill-conditioning and
large values of the penalty parameter $c_n$, the control inputs were rescaled as $\widehat{u} = (1/m) u$, which
corresponds to putting $m = 1$, $\underline{u} = - 2/3$, and $\overline{u} = 2/3$ for the original problem. 

\begin{figure}[t!]
\centering
\includegraphics[width=0.7\linewidth]{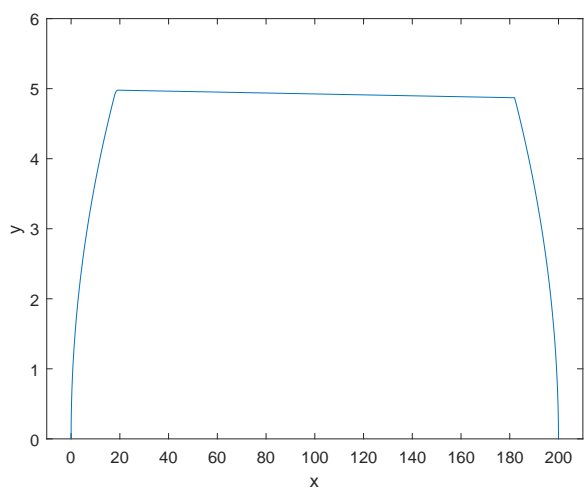}
\caption{The ``optimal'' tachogram.}
\label{fig:Tachogram}
\end{figure}

Let us discuss the results of the numerical experiment. The method terminated after 14 iterations by finding a feasible
critical point of the corresponding discrete optimal control problem. The computation time was 23 min and 39 s.
The ``optimal'' tachogram that assigns to each position of the train the computed ``optimal'' velocity is
given in Figure~\ref{fig:Tachogram}.

As in the previous example, the penalty parameter was increased only once to ensure sufficient decay of the
infeasibility measure on Step 3 of iteration 8. After this, a feasible point was computed and all subsequent iterations
remained feasible. Thus, the method solved 15 convex penalised subproblems \eqref{prob:StPen_PenaltySubproblem} and 8
optimal feasibility subproblems \eqref{prob:StPen_FeasibilitySubproblem}.
\end{example}


\bibliographystyle{abbrv}  
\bibliography{ExactPenaltyDC_bibl}

\end{document}